\documentclass[10pt,a4paper]{article}
\usepackage[margin=3cm]{geometry}
\usepackage{amsfonts}
\usepackage{amsmath}
\usepackage{graphicx}
\usepackage{epstopdf}
\usepackage{amsthm}
\usepackage{subfig}
\usepackage{algorithm}
\usepackage{algorithmic}

\usepackage{tikz}
\usetikzlibrary{positioning}
\usetikzlibrary{calc}
\usepackage{pdflscape,afterpage,capt-of}
\usepackage{booktabs, multirow, array}
\usepackage[inline]{enumitem}
\usepackage{color}
\usepackage{acronym}
\usepackage{framed}
\usepackage{bm}
\usepackage{amsopn}
\definecolor{blue}{RGB}{57, 57, 235}
\definecolor{red}{RGB}{165, 0, 33}
\definecolor{green}{RGB}{57, 235, 57}
\usepackage{hyperref}
\hypersetup{
  colorlinks = true,
  allcolors = blue,
  urlcolor = red,
}

\usepackage{xcolor}
\usepackage{acronym}
\usepackage{framed}
\usepackage{mathrsfs}
\usepackage{bm}

\usepackage{cleveref}

\newtheorem{theorem}{Theorem}[section]

\newtheorem{proposition}[theorem]{Proposition}
\newtheorem{example}[theorem]{Example}
\newtheorem{remark}[theorem]{Remark}
\numberwithin{equation}{section}

\definecolor{refblue}{RGB}{26,13,171}
\definecolor{newblue}{RGB}{13,13,217}
\definecolor{newred}{RGB}{221,13,13}

\def\b{\bm b}
\def\e{\bm e}
\def\dd{\mathrm d}
\def\u{\bm u}

\def\f{\bm f}
\def\v{\bm v}
\def\d{\bm d}
\def\q{\bm q}
\def\m{\bm m}

\def\x{\bm x}
\def\y{\bm y}
\def\A{\mathcal A}
\def\B{\mathcal B}

\def\L{\mathcal L}
\def\I{\mathcal I}
\def\D{\mathcal D}

\def\X{\mathcal X}

\def\F{\mathcal F}
\def\P{\mathcal P}
\def\W{\mathcal W}

\def\M{\mathcal M}

\def\T{\mathcal T}

\def\G{\mathcal G}
\def\z{\bm z}

\def\fvc{\mathbf V^{\mathrm c}}
\def\fvs{\mathbf V^{\mathrm {st}}}
\def\ftc{\mathbf T^{\mathrm c}}
\def\fts{\mathbf T^{\mathrm {st}}}

\newcommand{\eqdef}{:=}

\DeclareMathOperator*{\argmin}{arg \, min}

\DeclareMathOperator{\diag}{Diag}
\DeclareMathOperator{\rge}{rge}
\newcommand{\vertiii}[1]{{\left\vert\kern-0.25ex\left\vert\kern-0.25ex\left\vert #1
		\right\vert\kern-0.25ex\right\vert\kern-0.25ex\right\vert}}

	\title{\Large\bf An Efficient Augmented Lagrangian Framework for Dynamic Optimal Transport on Surfaces Based on Second-Order Cone Programming Reformulation\thanks{This work was supported by
the National Key R \& D Program of China (No. 2021YFA001300) and 
the National Natural Science Foundation of China (No. 12271150).}}
	
	\author{
		Liang Chen\thanks{School of Mathematics, Hunan University, Changsha, 410082, China
			(\url{chl@hnu.edu.cn}).}
		\quad
		Youyicun Lin\thanks{School of Mathematics, Hunan University, Changsha, 410082, China
			(\url{linyouyicun@hnu.edu.cn}).}
		\quad
		Yuxuan Zhou\thanks{Department of Mathematics, Southern University of Science and Technology, Shenzhen, 518055, China            (\url{zhouyx8@mail.sustech.edu.cn}).}}
	
	\date{\today}

\begin{document}
	\maketitle
	\begin{abstract}
		This paper proposes an efficient numerical optimization framework for solving dynamic optimal transport (DOT) problems on surfaces, computing both the quadratic Wasserstein distance and the associated interpolation.
		Building on the convex DOT model of Benamou-Brenier-Lisini, we first properly reformulate its dual problem, discretized on a triangular mesh in space and a staggered grid in time, into a linear second-order cone programming (SOCP) problem.
		Then the resulting  SOCP is solved via an inexact proximal augmented Lagrangian method with a highly efficient numerical implementation, and the algorithm is guaranteed to converge to a Karush-Kuhn-Tucker point without imposing any additional assumptions.
		Finally, we implement the proposed framework as an open-source software package.
		The effectiveness, robustness, and computational efficiency of the software are validated through extensive numerical experiments across diverse datasets, demonstrating that it consistently outperforms state-of-the-art surface DOT solvers by several times in speed, while the commercial solvers {\tt Gurobi} and {\tt MOSEK} either fail to solve the same SOCP reformulation due to out-of-memory or require substantially prolonged computation times.


\medskip
\noindent
{\bf MSCcodes:}
49Q22, 65K10, 90C06, 90-04
\end{abstract}

	\section{Introduction}
	Optimal transport (OT) traces its origins to Monge's pioneering work \cite{monge} and has grown into a cornerstone of pure and applied mathematics, profoundly shaped by Kantorovich's relaxation \cite{Kantorovich42,Kantorovich48}.
	Generally, it finds the most cost-efficient transportation plan between two probability measures, yielding the well-known Wasserstein distance.
	In the seminal work \cite{bb00} of Benamou and Brenier, a continuum-mechanics perspective was developed that reformulated the Monge-Kantorovich OT problem for measures with compact support in the same Euclidean space under a quadratic cost as a kinetic energy minimization problem constrained by a continuity equation, commonly referred to as \emph{dynamic optimal transport} (DOT). 
	Later, Lisini \cite{lisini07} extended DOT to Borel measures with finite moment on separable Banach spaces, and to smooth finite-dimensional complete Riemannian manifolds. 
	As a refined OT model, DOT gives not only the cost but also a time-dependent displacement interpolation, capturing realistic transport phenomena (e.g., fluid dynamics, crowd movement) in which intermediate states are physically meaningful \cite{Santamot}.

	As OT and DOT are fundamental optimization problems with growing practical applications, their computational and numerical aspects have drawn considerable interest in recent years \cite{peyre}.
	To date, highly efficient numerical methods for OT have been achieved 
	(see \cite{sinkhorn,solomon,chambolleot,hot} for example).
	In contrast, solving the associated DOT problems remains a challenge, especially on complex non-Euclidean domains (see \Cref{relatedworkds} for details).
	To bridge this methodological gap, this paper aims to design an efficient numerical optimization framework, together with an open-source software package\footnote{\url{https://github.com/chlhnu/DOTs-SOCP}}, for solving DOT problems on smooth surfaces, i.e., compact and connected two-dimensional Riemannian manifolds (see \Cref{sec:mdot} for the detailed background and concrete formulation of the problem).

	\subsection{Related works}
	\label{relatedworkds}
	Numerical investigation of DOT in flat domains of $\mathbb{R}^n$ was initiated by Benamou and Brenier \cite{bb00}, and has since inspired many computational approaches.
	These include interior-point methods \cite{nata21}, primal-dual splitting methods \cite{papa14,nata22,jose22}, and the accelerated proximal gradient method \cite{yjj24}.
	A recent work \cite{dotsoc} reformulated the dual of this DOT problem on staggered finite-difference grids as a linear second-order cone programming (SOCP) problem \cite{monteiro,soc,socsurvey},  then solved via an inexact proximal augmented Lagrangian method \cite{cl21}, built on a series of foundational works \cite{hestenes,powell,rock-alm,lxd16,lxdsgs}.
	As demonstrated in \cite{dotsoc}, this approach maintains superior scalability and computational efficiency compared to previous methods.

	On surfaces, \cite{hugo18} discretized the dual DOT problem on triangular meshes and solved the resulting convex optimization by the alternating direction method of multipliers (ADMM) \cite{glowinski75,gabay76}.
	Compared with the convolution method in \cite{solomon}, it produces interpolation densities with less diffusion. 
	As concluded in \cite[Section 7]{hugo18}, accelerating their method as much as possible is a desirable research topic, since their computational efficiency for large-scale triangular meshes was unsatisfactory.
	More recently, the accelerated proximal gradient method was implemented in \cite{yjjmanifold} for mean-field games (including DOT). Still, it requires a positive lower bound for both the initial and target measures.
	According to \cite[Section 5]{yjjmanifold}, it is hard to achieve a relative error of $10^{-4}$ for the residual of the Karush–Kuhn–Tucker (KKT) system due to the large-scale discretization and the complicated mesh.
	Given these limitations, DOT on surfaces remains a difficult problem in computational optimization.
	
	\subsection{Motivation and obstacles}
	Motivated by these related works, especially the compelling numerical performance \cite{dotsoc} for DOT in flat domains, it is natural to consider extending it to DOT on surfaces. 
	For this purpose, we should take a further look at \cite{dotsoc}, in which the simple shifted copy operators \cite[Eq. (3.4)]{dotsoc} are sufficient to play a decisive role in decoupling the quadratic inequality constraints from the discretized problem \cite[Eq. (2.8)]{dotsoc}, thanks to the uniform grid spacing. 
	This leads to the crucial decoupled problem \cite[Eq. (3.7)]{dotsoc}, in which the quadratic constraints were further uniformly converted to involve the second-order cone in $\mathbb{R}^{4D+2}$ with $D$ being the dimension of the spatial domain.
	Then, for the resulting SOCP reformulation \cite[Eq. (3.10)]{dotsoc}, the inexact proximal ALM \cite[Algorithm 1]{dotsoc} succeeds because the subproblems are conquered via the following three steps:
	\textbf{(i)} a discrete Poisson equation tackled by fast Fourier transform with a log-linear complexity; 
	\textbf{(ii)} projections to second-order cones admitting a linear complexity, whose advantage over existing approaches (solving cubic equations) was clearly demonstrated in \cite[Table 3.1]{dotsoc}; 
	and 
	\textbf{(iii)} a large-scale linear system with a diagonal coefficient matrix, which is an essential theoretical result proved by \cite[Proposition 3.2]{dotsoc}. 
	Besides, the existence of KKT points for both the original discretized dual problem and the SOCP reformulation is established without any additional assumptions, and a KKT solution to the former can be extracted directly from one of the latter \cite[Propositions 3.1 \& 3.3]{dotsoc}. 
	These results are based on the properties of the discrete gradient and divergence operators, determined only by the mesh structure. 
	Overall, the efficacy and efficiency of \cite{dotsoc} can be viewed as a combination of convergence guarantee and computationally economical subproblems, all depending on the staggered grids. 
	
	To extend the SOCP reformulation and efficient implementation from \cite{dotsoc} to DOT problems on smooth surfaces discretized with triangular meshes, an essential issue arises due to the irregular nature of the mesh.
	Specifically, the quadratic constraints (see Eq. \eqref{eq:dsP}) associated with a mesh vertex $\v$ contain discretized gradient quantities indexed by the triangles incident to $\v$ (denoted by $T_{\v}$), and weighted by the corresponding triangle areas and the area of the barycentric dual cell of $\v$ (see \Cref{fig:mesh}). 
	For different vertices, both the number of incident triangles and the corresponding weights vary.
	This is fundamentally different from staggered finite-difference grids (see \cite[Eq. (2.8)]{dotsoc}), forcing shifted copy operators no longer valid for decoupling the quadratic constraints. 
	As a consequence, when these constraints are decoupled (see Eq. \eqref{eq:tildep}), the resulting quadratic inequalities are not about the original variables directly, but involve a composition with extra linear transforms. 
	Thus, casting the decoupled problem as an SOCP formulation similar to \cite[Eq. (3.10)]{dotsoc} is not possible due to this composition, and the algorithmic design in \cite{dotsoc}, which is heavily tailored to the decoupling operator (shifted copy) and the resulting SOCP reformulation (admitting very simple subproblems), is not guaranteed to be applicable or efficient here. 
	Besides, whether the unconditional existence of KKT points holds for the discretized problem and the SOCP reformulation, along with the theoretical interconnections between them, requires a different route for examination.
	
	\subsection{Contributions}
	We first introduce vertex-dependent linear operators (see Eq. \eqref{eq:tvkv}),
	adapted to triangular meshes, to decouple the quadratic inequality constraints in the discretized dual DOT problem on surfaces (given by \eqref{eq:conti-ot-pri}, discretized on a triangular mesh in space and a staggered grid in time).
	These decoupled inequalities are further rearranged into second-order cone constraints
	indexed by staggered time grids and mesh vertices, where the cone associated with each vertex $\v$ has dimension $2+6|T_{\v}|$.
	A distinctive property of the resulting SOCP reformulation \eqref{eq:opt-ddot-soc}, compared with that for DOT in flat domains (see \cite[Eq. (3.10)]{dotsoc}), is that the conic constraint here is imposed on a linear transform of the auxiliary variable (i.e., ${\cal T}\z$), rather than directly on the auxiliary variable $\z$.
	Although this feature hinders the direct application of the algorithm used in \cite{dotsoc} and inheriting its efficient implementation, we elaborately designed an economical implementation of the inexact proximal ALM from \cite{cl21} to efficiently solve it. 
	Extensive numerical experiments demonstrated that the proposed framework is far superior to both state-of-the-art surface DOT solvers and commercial SOCP solvers. 
	The main contributions of this work are categorized as follows:
	\paragraph{Theory} 
	We prove the unconditional existence of KKT points for both the original discretized problem and the SOCP reformulation, and reveal that every KKT point of the reformulation intrinsically yields a KKT point of the original discretized problem (\Cref{prop:dotkkt,prop:soc-kkt}). 
	These results provide the fundamental theoretical guarantee for the convergence and efficacy of algorithms for the reformulation.

	\paragraph{Algorithm} 
	We design a tailored inexact proximal ALM (\Cref{alg:socinpalm}) to solve the SOCP reformulation. 
	A major advantage is that we factorize the extra linear operator $\cal T$ in \eqref{eq:opt-ddot-soc} as the product of a permutation matrix and a nonsingular diagonal matrix. 
	Then, there is no need to introduce an extra slack variable to deal with $\T$, and the core projection subproblem is decomposed into closed-form projections onto second-order cones of varying dimensions.

	\paragraph{Implementation} 
	By the SOCP reformulation and its essential property established in \Cref{prop:diag}, the resulting two subproblems in the adopted ALM algorithm are large-scale linear systems of equations, but not hard to solve, in which one reduces to independent surface Poisson systems with fixed coefficient matrices factorized in advance, and the other one admits a diagonal coefficient matrix.

	\paragraph{Software} 
	We provide an open-source Python software package with fully reproducible numerical experiments. 
	According to the numerical results, the proposed framework exhibits superior accuracy, robustness, and efficiency compared to both existing surface DOT solvers \cite{hugo18,yjjmanifold} and commercial conic programming solvers ({\tt Gurobi} \cite{gurobi} and {\tt MOSEK} \cite{mosek}) applied to the same SOCP reformulation.
	Specifically, it reliably attains the prescribed tolerances and is at least $7$ times faster than the compared surface DOT solvers,
	while the commercial solvers either fail to solve the same SOCP reformulation or require substantially longer computation times.

	\subsection{Organization}
	The remaining parts of this paper are organized as follows.
	\Cref{sec:pre} introduces the notation and preliminaries, including the DOT problems on a smooth surface embedded in $\mathbb{R}^3$ and the dual problem (\Cref{sec:mdot}), together with the triangular finite-element discretization (\Cref{sec:fem}).
	\Cref{sec:anal} establishes essential properties of the discretized dual DOT problem that are crucial for reformulation and designing optimization algorithms.
	In \Cref{sec:soc}, we reformulate the discretized dual DOT problem into a linear SOCP problem and analyze its essential properties.
	In \Cref{sec:alg}, we introduce an inexact proximal ALM to solve the reformulation, which is guaranteed to converge to a primal-dual solution pair without imposing any additional assumptions, together with the details of the efficient numerical implementation.
	In \Cref{sec:num}, we use extensive numerical experiments to validate the superior efficiency of the proposed approach compared with existing software packages and commercial solvers.
	Finally, we conclude this paper in \Cref{sec:con}.

	\section{Notation and preliminaries}
	\label{sec:pre}
	Throughout this paper, 
	$\mathbb{R}^n$ represents the $n$-dimensional real Euclidean space, and  $\mathbb{R}_+^n$ ($\mathbb{R}_-^n$) represents the non-negative (non-positive) orthant in $\mathbb{R}^n$.
	For a finite set $\G$, we denote by $|\G|$ its cardinality.
	For the two given (column) vectors $\x$ and $\y$, we denote $(\x;\y) := (\x^{\top},\y^{\top})^{\top}$, where $(\cdot)^\top$ means the transpose.
	We use $\e_{k}\in \mathbb{R}^{n}$ to represent the $k$-th unit vector in $\mathbb{R}^n$.
	The two vectors $\bm{1}_n$ and $\bm{0}_n$ represent the vectors in $\mathbb{R}^n$ whose entries are $1$ and $0$, respectively.
	The notation $\bm{x}\ge 0$ ($\bm{x}\le 0$, etc.) means that all elements of $\bm{x}$ are non-negative (non-positive, etc.).
	Let $\mathbb{K}_{\mathrm{soc}}:=\{\y\in\mathbb{R}^{1+n}\mid y_0\geq \|\bar{\y}\|\}$ be the second-order cone in $\mathbb{R}^{1+n}$, in which for
	a given vector $\y\in\mathbb{R}^{1+n}$, we use $y_0$ to denote its first component and define $\bar{\y}\eqdef(y_1,\dots,y_n)^{\top}$. 
	For simplicity, we denote $\y = (y_0;\bar{\y})$.
	
	We use $\rge(A)$ and $\ker(A)$ as the range and the null space of a given matrix $A$.
	For any given matrices $A$ and $B$, $A\otimes B$ denotes their Kronecker product, 
	and $A\odot B$ denotes their Hadamard product, provided they share the same size.
	We use $\I_n$ to denote the $n\times n$ identity matrix and use $\diag(D_1,\ldots,D_n)$ to denote a block-diagonal matrix whose diagonal blocks are $D_i$,  ordered from $i=1$ to $i=n$.

	Given a two-dimensional Riemannian manifold $(\X,g)$, where $\X \subset\mathbb{R}^3$ is a smooth surface and $g$ denotes the induced Riemannian metric, we use $d_g(\cdot,\cdot)$ to denote the geodesic distance on $\X$. 
	Let $C(\X)$ be the space of continuous real-valued functions on $\X$ and denote the topological dual of $C(\X)$ by $\M(\X)$, which is identified with the space of finite signed Radon measures on $\X$. 
	We use $\M_+(\X)\subset \M(\X)$ to represent the set of all finite non‑negative Radon measures on $\X$, which is a closed convex cone in $\mathcal{M}(\X)$ with respect to both the total variation norm and the weak* topology (induced by $C(\X)$).
	For each $\x\in\X$, $T_{\x}\X$ denotes the tangent space at $\x$, equipped with the inner product $\langle\cdot,\cdot\rangle_{g(\x)}$ and the induced norm $\|\cdot\|_{g(\x)}$.
	Denote by $T\X$ the corresponding tangent bundle, and let $\pi: T\X \to \X$ be the canonical projection map.

	Set $\Omega:= [0,1] \times \X$ as the space-time cylinder equipped with the product topology and the corresponding Borel $\sigma$-algebra. 
	Let $C(\Omega; T\mathcal{X})$ be the space of continuous maps $\b : \Omega \to T\mathcal{X}$ such that for each $t \in [0,1]$, $\b(t, \cdot)$ is a section of $T\mathcal{X}$ (i.e., $\pi(\b(t, \boldsymbol{x})) = \boldsymbol{x}$), endowed with the supremum norm $\|\b\|_\infty = \sup_{(t,\x) \in \Omega} |\b(t,\x)|_g$, thereby structuring it as a Banach space.  
	In other words, for any given time $t$, the vector $\b(t,\x)$ is purely spatial and intrinsically tied to the tangent space at $\x$, while varying continuously with respect to both time $t$ and space $\x$. 
	By virtue of the induced Riemannian metric $g$, which canonically identifies the tangent and cotangent bundles, we naturally identify its continuous dual space by $\mathcal{M}(\Omega; T\X)$, which consists of 
	$T\mathcal{X}$-valued Radon measures on $\Omega$ with finite total variation.

	\subsection{DOT on smooth surfaces}
	\label{sec:mdot}
	
	Given measures $\rho_0,\rho_1 \in \M_+(\X)$ on $\X$ sharing the same total mass, the OT problem \cite{villaninew} is given by 
	\begin{equation}
		\label{eq:2-wass}
		\begin{array}{l}
			\displaystyle
			\min_{\gamma\in\Gamma(\rho_0,\rho_1)}\int_{\X\times \X} d_g(\x,\y)^2 \dd \gamma(\x,\y),
		\end{array}
	\end{equation}
	where $\Gamma(\rho_0,\rho_1)$ is the set of non-negative measures on $\X\times \X$ whose first and second marginals are $\rho_0$ and $\rho_1$, respectively. 
	According to \cite[Theorem 2.4]{hugo20}, the optimal value of the OT problem \eqref{eq:2-wass} is twice the optimal value of the following DOT problem
	\begin{equation}
		\label{eq:conti-ot-pri}
		\min\limits_{\rho\in\M(\Omega),
			\, 
			\m\in\M(\Omega;T\X)}
		\left\{\mathfrak{B}(\rho,\m)
		\ \Big\vert\ 
		\begin{array}{lr}
			\partial_t\rho + \mathrm{div}_{\x}(\m) = 0, 
			\ 
			&\rho(0,\cdot)=\rho_0,
			\\
			\langle \bm{m}, \bm{n} \rangle_{g(\boldsymbol{x})} = 0
			\text{ on } \partial\mathcal{X},
			\ 
			&\rho(1,\cdot)=\rho_1
		\end{array}
		\right\}.
	\end{equation}
	Here, the constraints in \eqref{eq:conti-ot-pri} are understood in the weak sense and $\boldsymbol{n}$ stands for the outward unit normal to $\partial\mathcal{X}$.
	%
	Meanwhile, the Benamou-Brenier functional $\mathfrak{B}\colon \M(\Omega)\times\M(\Omega; T\X)$ in \eqref{eq:conti-ot-pri} is defined by 
	\begin{equation*}
		\mathfrak{B}(\rho,\m)\eqdef\sup\limits_{a,\b}
		\Big\{\int_\Omega a(t,\x)\dd \rho(t,\x)+\int_\Omega\langle\b(t,\x),\dd\m(t,\x)\rangle_{g(\x)} \mid (a,\b)\in P
		\Big\},
	\end{equation*}
	where $P$ is the closed convex set defined by
	\begin{equation}
		\label{eq:calP}
		\begin{array}{ll}
			P\eqdef\Big\{(a,\b)\in C(\Omega)\times C(\Omega;T\X)\mid
			a(t,\x)+\frac{\|\b(t,\x)\|^2_{g(\x)}}{2}\leq 0 \quad \forall\, (t,\x)\in \Omega\Big\}.
		\end{array}
	\end{equation}
	By definition, the functional $\mathfrak{B}$ is nonnegative, convex, and weak-* lower semicontinuous, ensuring \eqref{eq:conti-ot-pri} is a well-defined linearly constrained convex optimization problem. 
	Moreover, $\mathfrak{B}(\rho,\m)<+\infty$ if and only if $\rho\in\M_+(\Omega)$ and there exists a $\rho$-measurable time-dependent tangent vector field $\bm r$, satisfying $\bm r(t,\x)\in T_{\x}\X$ for $\rho$-almost every $(t,\x) \in \Omega$, such that $\m = \bm r \rho$ and $\int_{\Omega}\|\bm{r}(t,\x)\|^2_{g(\x)}\dd\rho(t,\x)<+\infty$. 
	In this case, the DOT problem \eqref{eq:conti-ot-pri} minimizes the total kinetic energy over the continuous time interval $[0,1]$. 
	One may refer to \cite[Section 2.1]{hugo20} for the details and \cite[Proposition 5.18]{Santamot} for its Euclidean counterpart.
	 
	One of the most common numerical approaches is to apply ADMM \cite{glowinski75,gabay76} to the dual problem of \eqref{eq:conti-ot-pri}, as in \cite{bb00}.
	This approach was later extended in \cite{hugo18} to two-dimensional manifolds in $\mathbb{R}^3$.
	Following the discussion in \cite{hugo18} and \cite[Section 6]{Santamot}, the dual problem of \eqref{eq:conti-ot-pri} is given by
	\begin{equation}
		\label{eq:cont-dual-dot}
		\begin{array}{l}
			\sup\limits_{\phi \in C^1(\Omega) }  \Big\{\int_{ \X}\phi(1,\x) \dd \rho_1(\x) - \int_{\X}\phi(0,\x) \dd \rho_0(\x)
			\mid
			(\partial_t\phi,\nabla_{\x}\phi) \in P\Big\},
		\end{array}
	\end{equation}
	where
	$C^1(\Omega)$ denotes the space of continuously differentiable functions on $\Omega$, 
	$\nabla_{\x}\phi$ denotes the Riemannian gradient with respect to the spatial variable, and 
	$P$ is defined by \eqref{eq:calP}.
	In what follows, we focus on the dual problem \eqref{eq:cont-dual-dot} and develop an efficient numerical approach to solve its discretized problem.

	\subsection{Finite-element discretization}
	\label{sec:fem}
	This part briefly introduces the spatial discretization of the dual DOT problem \eqref{eq:cont-dual-dot}, together with a staggered grid for the time interval $[0,1]$.
	Such a discretization framework follows from \cite{surface,hugo18}.
	Specifically, the spatial domain $\X$ is approximated by a regular triangular mesh in the sense of \cite[Definition 4.4.13]{fem}, ensuring that for each triangle, the ratio of its inscribed circle diameter to its diameter is uniformly bounded from below by a positive constant. Furthermore, the intersection of any two distinct triangles is either empty, a common vertex, or a common edge.
	The resulting discrete surface consists of a set $V\subset \mathbb{R}^3$ of vertices and a set $T$ of (closed) triangles $\f$, each with its three vertices in $V$.
	We fix the orderings $V = \{\v_1,\ldots,\v_{|V|}\}$ and $T = \{\f_1,\ldots, \f_{|T|}\}$, and define the adjacent sets
	\begin{equation*}
		\begin{cases}
			V_{\f} : = \{\v\in V\mid \v \mbox{ is a vertex of the triangle } \f\}\quad \forall \f\in T, 
			\\
			T_{\v} : = \{\f\in T\mid \v \mbox{ is a vertex of the triangle } \f \} \quad \forall \v\in V. 
		\end{cases}
	\end{equation*}
	We use $|\f|$ to denote the area of the triangle $\f$, and define $|\v| := \frac{1}{3}\sum_{\f\in T_{\v}}|\f|$ as the area of the barycentric dual cell of $\v$, which is used to represent the ``area'' of the vertex $\v$   (see \Cref{fig:mesh} for an illustration). 
	\begin{figure}
		\centering
		\begin{tikzpicture}[node distance = 1mm and 1mm]
			\node (img) {
				\includegraphics[width = 0.3\textwidth]{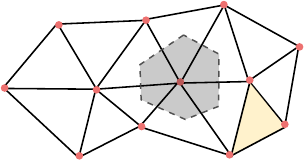}
			};
			\node[ xshift = 11pt, yshift = 7pt]{$\tilde{\v}$};
			\node[ xshift = 44pt, yshift = 10pt]{$\bar{\v}$};
			\node[ xshift = 44pt, yshift = -16pt]{$\hat{\f}$};
		\end{tikzpicture}
		\caption{
			For the two vertices (red points)
			$\bar{\v}$ and $\tilde{\v}$, the triangle $\hat{\f}$, whose area is denoted by $|\hat{\f}|$, belongs to the set $T_{\bar{\v}}$, and the vertex $\bar{\v}$ belongs to the set $V_{\hat{\f}}$. The gray region indicates the barycentric dual cell of $\tilde{\v}$, whose area is denoted by $|\tilde{\v}|$. 
		}
		\label{fig:mesh}
	\end{figure}
	For the time interval $[0,1]$, we divide it into $N$ segments.
		Then, we define the centered and staggered time index sets by
		\begin{equation*}
			\begin{array}{l}
				\G_{\text{time}}^{\text{c}}:=\{k \mid k = 0,\ldots,N\} \quad\mbox{and}\quad
				\G_{\text{time}}^{\text{st}}:=\big\{k \mid k= \frac{1}{2},\frac{3}{2},\ldots, N-\frac{1}{2}\big\}.
			\end{array}
		\end{equation*}
		For each $k\in\G_{\rm time}^{\rm c}\cup\G_{\rm time}^{\rm st}$, the corresponding time point is $k/N$.
		Since the discrete variables we will use are indexed by different combinations of time grids, vertices, and triangles, we fix the product index sets 
		\begin{equation*}
			\mathbf V^{\rm c}:=\mathcal G^{\rm c}_{\rm time}\times V,\quad
			\mathbf V^{\rm st}:=\mathcal G^{\rm st}_{\rm time}\times V,\quad
			\mathbf T^{\rm c}:=\mathcal G^{\rm c}_{\rm time}\times T,
			\quad
			\mbox{and}
			\quad
			\mathbf T^{\rm st}:=\mathcal G^{\rm st}_{\rm time}\times T .
		\end{equation*}
		These index sets are ordered lexicographically with the time index $k$ first, as illustrated in \Cref{tab:illustration}, in which the two spaces $\mathbb R^{m\times\ftc}$ and $\mathbb R^{m\times{\fts}}$ will be used frequently. Typically, we set $m=3$, and omit $m$ when $m=1$. 
		\begin{table}[H]
			\centering
			\caption{Illustration of vectors in time-vertex indexed and time-triangle indexed product spaces.}
			\scriptsize
			\begingroup
			\setlength{\tabcolsep}{2.2pt}
			\renewcommand{\arraystretch}{1.5} 
			\begin{tabular}{|c|c|c|}
				\hline
				{\bf Space} & {\bf Block vector} & {\bf Components} 
				\\
				\hline
				$\mathbb R^{\fvc}$ ~& 
				$\x=(\x_k)_{k\in\G_{\rm time}^{\rm c}} \in \mathbb R^{|\fvc|}$ &
				\multirow{2}{*}{$\x_k:=(x_{k,\v_1};\ldots;x_{k,\v_{|V|}})\in\mathbb R^{|V|}$}
				\\
				\cline{1-2}
				$\mathbb R^{\fvs}$ & 
				$\x=(\x_k)_{k\in\G_{\rm time}^{\rm st}} \in \mathbb R^{|\fvs|}$ &
				\\
				\hline
				$\mathbb R^{m\times\ftc} ~$ &
				$\y=(\y_k)_{k\in\G_{\rm time}^{\rm c}} \in \mathbb R^{m |\ftc|}$ &
				$\y_k:=(\y_{k,\f_1};\ldots;\y_{k,\f_{|T|}}) \in\mathbb R^{m |T|}$
				\\
				\cline{1-2}
				$\mathbb R^{m\times{\fts}}$ &
				$\y=(\y_k)_{k\in\G_{\rm time}^{\rm st}} \in \mathbb R^{m |\fts|}$ &
				\quad \text{with each} $\y_{k,\f_i}\in\mathbb R^m$
				\\
				\hline
			\end{tabular}
			\endgroup
			\label{tab:illustration}
		\end{table}
		Then, for constructing the discrete counterpart of $\phi \in C^1( \Omega)$ in \eqref{eq:cont-dual-dot}, we define $\bm\varphi\in \mathbb{R}^{\fvc}$ by ${\varphi}_{k,\v} := \phi(k/N,\v)$ with $(k,\v) \in \fvc$.
		Thus $\bm\varphi_{k} \in\mathbb{R}^{|V|}$ denotes the sub-vector of nodal values at the time index $k$.
		To approximate the spatial Riemannian gradient $\nabla_{\x} \phi(k/N,\cdot)$ on a triangle $\f$, we denote its three vertices by $\v_1^{\f}$, $\v_2^{\f}$, and $\v_3^{\f}$. 
		Let $h_{\v}$, $\v\in V$, be the standard hat function (piecewise linear finite-element basis), which is linear on each triangular plane and satisfies $h_{\v}(\v) =1 $ and $h_{\v}(\v') = 0$ for all $\v' \in V \setminus\{\v\}$.
		Utilizing a standard parametrization approach (see \Cref{sec:gradient}), we obtain
		\begin{equation}
			\label{eq:nabla-h}
			(\nabla h_{\v_1^{\f}},\nabla h_{\v_2^{\f}},\nabla h_{\v_3^{\f}}) = (J_{\f}^{\top})^\dag\begin{pmatrix}
				-1 &1 &0 \\
				-1 & 0& 1
			\end{pmatrix}\ 
			\mbox{with }  
			J_{\f} = (\v^{\f}_2-\v^{\f}_1,\v^{\f}_3-\v^{\f}_1)\in \mathbb{R}^{3\times 2},
		\end{equation}
		where  $(J_{\f}^{\top})^\dag:= J_{\f}(J_{\f}^{\top}J_{\f})^{-1}$ is the Moore-Penrose pseudo-inverse of $J_{\f}^{\top}$. 
		Based on \eqref{eq:nabla-h}, define the discrete spatial difference operator $\tilde \A_s:\mathbb{R}^{|V|}\rightarrow \mathbb{R}^{3\times {T}}$ by
		\begin{equation}
			\label{eq:As}
			\begin{array}{ll}
				(\tilde \A_s\bm\gamma)_{\f}:= 
					\sum\limits_{\v\in V_{\f}} \gamma_{\v} \nabla h_{\v}\in \mathbb{R}^3,
					\quad \bm\gamma\in \mathbb{R}^{|V|},
					\qquad 
					\bm{f}\in T.
			\end{array}
		\end{equation}
		Then, for each $k\in\G^{\rm c}_{\rm time}$, $\tilde \A_s \bm{\varphi}_k$
		approximates $\nabla_{\x}\phi(k/N,\cdot)$. 
		Next,  define the discrete temporal difference operator 
			$\tilde \A_t\colon\mathbb{R}^{|\G_{\text{time}}^{\text{c}}|}\rightarrow \mathbb{R}^{|\G_{\text{time}}^{\text{st}}|}$ by
			\begin{equation*}
				\begin{array}{ll}
					(\tilde \A_t\bm\zeta)_k:=N(\zeta_{k+\frac{1}{2}}-\zeta_{k-\frac{1}{2}}),
					\quad
					k= \frac{1}{2},\frac{3}{2},\ldots, N-\frac{1}{2}.
				\end{array}
		\end{equation*}
		Combining the spatial-temporal parts above, we obtain the discrete gradient operator $\A:\mathbb{R}^{\fvc}\to \mathbb{R}^{\fvs}\times \mathbb{R}^{3\times\ftc}
			$, defined by
		\begin{equation}
			\label{defA}
			\A =\begin{pmatrix}
				\A_t\\
				\A_s
			\end{pmatrix}
			\quad\mbox{with}\quad
			\A_t:= \tilde \A_t\otimes \I_{|V|}
			\quad\mbox{and}\quad
			\A_s:=    \I_{|\G^{\text{c}}_{\text{time}}|}\otimes \tilde \A_s.
		\end{equation}

		Let $\q \equiv (\bm A;\bm B) \in \mathbb{R}^{\fvs}\times \mathbb{R}^{3\times\ftc}$ be the auxiliary variable defined by $\bm A := \A_t \bm\varphi$ and $\bm B := \A_s\bm\varphi$.
		Here, $\bm A$ is indexed by $\fvs$, while $\bm B$ is indexed by $\ftc$, with each component $\bm B_{k,\f}\in\mathbb{R}^3$.
		This mismatch between index sets prevents a direct discretization of the constraint set $P$ given by \eqref{eq:calP} in the dual DOT problem \eqref{eq:cont-dual-dot}.
		To address this issue, and for the convenience of the forthcoming analysis, we define the time interpolation operator $\L_{t}\colon \mathbb{R}^{\fvc}\rightarrow \mathbb{R}^{\fvs}$ and the space interpolation operator $
			\mathcal{L}_{s} \colon\mathbb{R}^{\fvc}\rightarrow \mathbb{R}^{\ftc}$ by
		\begin{equation}
			\label{eq:li}
			\begin{cases}
				(\L_{t}\bm\varphi)_{k,\v}:= \frac{\varphi_{k+\frac{1}{2},\v}+\varphi_{k-\frac{1}{2},\v}}{2}, & k \in \G_{\text{time}}^{\text{st}},
				\\[.2em]
				(\mathcal{L}_{s} \bm\varphi)_{k,\f}:=\frac{1}{3}\sum_{\v\in V_{\f}}\varphi_{k,\v}, & k \in\G_{\text{time}}^{\text{c}}.
			\end{cases}
		\end{equation}
		To discretize the inner product $\langle\cdot,\cdot\rangle_{g(\x)}$ of $\cal X$ that incorporates the geometric information, we set the weight matrices and define
		\begin{equation}
			\label{eq:weighted-matrix}
			\begin{array}{ll}
				\W_V:=\diag(|\v_1|,\ldots,|\v_{|V|}|),
				&
				\W_T:=\diag(|\f_1|,\ldots,|\f_{|T|}|),
				\\[1mm]
				\W_{\mathrm{c},V}:=\frac{1}{N}\I_{|\G^{\text{c}}_{\text{time}}|}\otimes \W_V, 
				&
				\W_{\mathrm{c},T}:=\frac{1}{N}\I_{|\G^{\text{c}}_{\text{time}}|}\otimes \W_T\otimes \I_3, 
				\\[1mm]
				\W_{\mathrm{st},V}:=\frac{1}{N}\I_{|\G^{\text{st}}_{\text{time}}|}\otimes \W_V,
				&
				\W_{\mathrm{st},T}:=\frac{1}{N}\I_{|\G^{\text{st}}_{\text{time}}|}\otimes \W_T\otimes \I_3.
			\end{array}
		\end{equation}
		Then, we define
		\begin{equation}
			\label{eq:inner-tv}
			\begin{array}{ll}
				\langle\bm x ,\hat{\bm x }\rangle_{\mathrm{c},V}:= 
				\langle\bm x ,\W_{\mathrm{c},V} \hat{\bm x }\rangle
				\  \,
				\forall \bm x,\hat{\bm x} \in \mathbb{R}^{\fvc}, 
				&
				\langle\bm y ,\hat{\bm y }\rangle_{\mathrm{c},T}:= 
				\langle\bm y ,\W_{\mathrm{c},T} \hat{\bm y }\rangle
				\ \,
				\forall \bm y,\hat{\bm y} \in \mathbb{R}^{3\times\ftc}, 
				\\
				\langle\bm x ,\hat{\bm x }\rangle_{\mathrm{st},V}:= 
				\langle\bm x ,\W_{\mathrm{st},V} \hat{\bm x }\rangle
				\ \,
				\forall \bm x,\hat{\bm x} \in \mathbb{R}^{\fvs}, 
				&
				\langle\bm y ,\hat{\bm y }\rangle_{\mathrm{st},T}:= 
				\langle\bm y ,\W_{\mathrm{st},T} \hat{\bm y }\rangle
				\  \,
				\forall \bm y,\hat{\bm y} \in \mathbb{R}^{3\times\fts}.
			\end{array}
		\end{equation}
		Then, the adjoint operators of $\L_{t}$ and $\L_{s}$ are defined from 
		\begin{equation}
			\label{adjoint}
			\begin{array}{ll}
				\langle \bm\varphi,\L_{t}^*\bm\psi\rangle_{\mathrm{c},V}=\langle \L_{t}\bm\varphi,\bm\psi\rangle_{\mathrm{st},V} 
				\ \mbox{and}
				\ 
				\langle \bm\varphi,\L_{s}^* \y\rangle_{\mathrm{c},V}=
				\tfrac{1}{N}\sum_{(k,\f)\in\ftc}|\f|(\L_s\bm{\varphi})_{k,\f} y_{k,\f}.
			\end{array}
		\end{equation}
		Specifically, for $\y \in \mathbb{R}^{\ftc}$, one has 
		\begin{equation}
			\label{eq:lxadj}
			\begin{array}{ll}
				(\mathcal{L}_{s}^*\y)_{k,\v}: = \frac{1}{3|\v|}
				\sum\limits_{\f\in T_{\v}} |\f| y_{k,\f}.
			\end{array}
		\end{equation}
		
		Based on the above discussions, the set $P$ in \eqref{eq:calP} is discretized as
		\begin{equation}
			\label{eq:dsP}
			\mathbb{P} := \Big\{\bm{q}\equiv (\bm{A};\bm{B}) \in \mathbb{R}^{\fvs}\times \mathbb{R}^{3\times{\ftc}} \ \big\vert \  A_{k,\v} + \tfrac{1}{2} \big(\L_{t} \L^*_{s} (\vertiii{\bm{B}}^2)\big)_{k,\v}\leq 0 \ \ \forall \, (k,\v)\in \fvs\Big\},
		\end{equation}
		where $\vertiii{\bm{B}}^2\in\mathbb{R}^{\ftc}$ is defined by
		$(\vertiii{\bm{B}}^2)_{k,\f} := \|\bm{B}_{k,\f}\|^2$ with $\bm{B}_{k,\f}\in\mathbb{R}^3$.
		As a consequence, the discretized problem of \eqref{eq:cont-dual-dot} on the discrete surface is given by the following minimization reformulation
		\begin{equation}
			\label{eq:opt-ddot}
			\min_{\bm{\varphi},\q}
			\left\{\langle \bm\varphi,\bm{c}\rangle_{\mathrm{c},V}:=\sum_{\v \in V}|\v| \varphi_{0,\v} \mu_{0,\v}-\sum_{\v \in V}|\v| \varphi_{N,\v} \mu_{N,\v}
			\ \Big\vert\
			\begin{array}{ll}
				\mathcal{A}\bm{\varphi} = \q\equiv (\bm{A};\bm{B}),\\
				\bm{A} + \frac{1}{2} \L_{t} \L^*_{s} (\vertiii{\bm{B}}^2)\le 0
			\end{array}
			\right\},
		\end{equation}
		where $\bm\mu_0,\bm\mu_N\in \big\{\bm\mu\in\mathbb{R}_+^{|V|}\mid \sum_{\v\in V} |\v|\mu_{\v} =1  \big\}$ are the normalized discrete representations of  $\rho_0$ and $\rho_1$.
		For each $\v\in V$, the quantities
		$|\v|\mu_{0,\v}$ and $|\v|\mu_{N,\v}$ are decided by the masses of $\rho_0$ and $\rho_1$ on the barycentric dual cell associated with $\v$.
		
		\begin{remark}
			The above discretization scheme follows the framework established in \cite{hugo18}. Consequently, under the mesh regularity and consistency assumptions used in \cite[Section 4]{hugo20}, the interpolated solutions of the discretized problems converge to the solution of the continuous problem as the discretization is progressively refined.
		\end{remark}

		\section{Analysis on discretized dual DOT problem}
		\label{sec:anal}
		
		In this section, we formulate the KKT system of the discretized problem \eqref{eq:opt-ddot} and establish that it always admits a solution, an essential property safeguarding the convergence of algorithms. 
		
		With the dual variable
		$\bm\lambda \equiv (\bm\lambda_0,\bar{\bm{\lambda}}, \bm\lambda_q)
		\in
		\mathbb{R}^{\fvs}
			\times
			\mathbb{R}^{3\times{\ftc}}
			\times
			\mathbb{R}^{\fvs}$, the Lagrangian function for \eqref{eq:opt-ddot} is given by
		\begin{equation}
			\label{eq:dot-lag}
			\begin{array}{ll}
				\mathfrak{L}(\bm{\varphi},\q;\bm\lambda) =
				\langle \bm\varphi,\bm{c}\rangle_{\mathrm{c},V} + \langle \bm\lambda_0,\A_t\bm\varphi-\bm A\rangle_{\mathrm{st},V} 
				\\
				\hspace{9.8em}+\langle\bar{\bm\lambda},\A_s\bm\varphi-\bm B\rangle_{\mathrm{c},T}-\langle  \bm\lambda_q, \bm{A} + \frac{1}{2} \L_{t} \L^*_{s} (\vertiii{\bm{B}}^2)\rangle_{\mathrm{st},V},
			\end{array}
		\end{equation}
		where $\bm\varphi\in \mathbb{R}^{\fvc}$
		and
		$\q \equiv (\bm A;\bm B) \in \mathbb{R}^{\fvs}\times\mathbb{R}^{3\times{\ftc}}$.
		
		Let $\A^* = (\A^*_t,\A^*_s)$ be the Euclidean adjoint of the discrete gradient operator $\A$ defined in \eqref{defA}.
		It is easy to see that
		\begin{equation}\nonumber
			\begin{array}{ll}
				\langle \A_s\bm\varphi,\bar{\bm\lambda}\rangle_{\mathrm{c},T}
				= \frac{1}{N}\sum\limits_{(k,\f)\in \ftc}|\f| \big(\sum_{\v\in V_{\f}}\varphi_{k,\v}\nabla h_{\v}\big)^{\top}\bar{\bm\lambda}_{k,\f}
				\\[3mm]
				=\frac{1}{N}\sum\limits_{(k,\v)\in \fvc} \sum\limits_{\f\in T_{\v}}
				\varphi_{k,\v}|\f|\big(\nabla h_{\v}^{\top}\bar{\bm\lambda}_{k,\f}\big)
				=\frac{1}{N}\sum\limits_{(k,\v)\in \fvc}
				|\v|
				\varphi_{k,\v}
				\big(\sum\limits_{\f\in T_{\v}}\frac{|\f|}{|\v|}(\nabla h_{\v}^{\top}\bar{\bm\lambda}_{k,\f})
				\big)
				\\[4mm]
				\displaystyle
				= \langle\bm\varphi,\W_{\mathrm{c},V}^{-1}\A_s^* \W_{\mathrm{c},T}\bar{\bm\lambda} \rangle_{\mathrm{c},V}.
			\end{array}
		\end{equation}
		Similarly, one has that
		$\langle \A_t\bm\varphi,\bm\lambda_0\rangle_{\mathrm{st},V} = \langle \bm\varphi,\W_{\mathrm{c},V}^{-1}\A_t^* \W_{\mathrm{st},V}\bm\lambda_0\rangle_{\mathrm{c},V}$.
		Taking the above two equalities to \eqref{eq:dot-lag} and using \eqref{adjoint}, one has
		\begin{equation}
			\label{alfnew}
			\begin{array}{ll}
				\mathfrak{L}(\bm\varphi,\q;\bm\lambda)
				=
				\big\langle \bm\varphi,\bm{c} + \W_{\mathrm{c},V}^{-1}\A_t^* \W_{\mathrm{st},V}\bm\lambda_0 + \W_{\mathrm{c},V}^{-1}\A_s^* \W_{\mathrm{c},T}\bar{\bm\lambda}\big\rangle_{\mathrm{c},V}
				\\
				\hspace{10em}
				- \langle \bm\lambda_0+\bm\lambda_q,\bm A\rangle_{\mathrm{st},V}-\langle\bar{\bm\lambda},\bm B\rangle_{\mathrm{c},T}-\langle  \bm\lambda_q,  \frac{1}{2} \L_{t} \L^*_{s} (\vertiii{\bm{B}}^2)\rangle_{\mathrm{st},V}.
			\end{array}
		\end{equation}
		So the KKT system of \eqref{eq:opt-ddot} is given by
		\begin{equation}
			\label{eq:ddot-kkt}
			\begin{cases}
				\A\bm{\varphi}=\q \equiv (\bm A;\bm B), 
				\  
				\W_{\mathrm{c},V}\bm{c}+\A^*(\W_{\mathrm{st},V}\bm\lambda_0;\W_{\mathrm{c},T}\bar{\bm{\lambda}})=0,
				\ 
				\bm\lambda_0+\bm\lambda_q = 0,
				\\[1mm]
				0\leq\bm\lambda_0\perp \bm{A}+ \frac{1}{2} \L_{t} \L^*_{s} \big(\vertiii{\bm{B}}^2\big)\leq 0, 
				\ 
				\bar{\bm{\lambda}} = \big( (\L_{s}\L^*_{t}\bm\lambda_0)\otimes\bm{1}_3\big)\odot \bm{B}.
			\end{cases}
		\end{equation}
		Moreover, we have the following result on the unconditional existence of solutions to the KKT system \eqref{eq:ddot-kkt}.
		\begin{theorem}
			\label{prop:dotkkt}
			The solution set to the KKT system \eqref{eq:ddot-kkt} is not empty, i.e., \eqref{eq:opt-ddot} has both optimal solutions and Lagrange multipliers.
		\end{theorem}

		\begin{proof}
			By directly calculating from \eqref{alfnew}, the objective function of the dual problem to \eqref{eq:opt-ddot} is given by  
			\begin{equation*}
				\begin{array}{ll}
					\inf\limits_{\bm\varphi,\bm A, \bm B} \mathfrak{L}(\bm\varphi,\q;\bm\lambda)
					\\[2mm]
					=
					\begin{cases}
						\inf\limits_{\bm B}\Xi(\bm B)
						&
						\mbox{if}\ 
						\W_{\mathrm{c},V}\bm{c}+\A^*(\W_{\mathrm{st},V}\bm\lambda_0;\W_{\mathrm{c},T}\bar{\bm{\lambda}})=0,
						\ 
						\bm\lambda_q + \bm\lambda_0=0,
						\ \L_{s}\L^*_{t}\bm\lambda_q\leq 0,
						\\
						-\infty &\mbox{otherwise},
					\end{cases}
				\end{array}
			\end{equation*}
			where $\Xi(\bm B):=
			-\langle\bar{\bm\lambda},\bm B\rangle_{\mathrm{c},T}  
			+ {\frac{1}{2}\langle \bm\lambda_0,\L_t\L_s^*\vertiii{\bm B}^2\rangle_{\mathrm{st},V}}$. 
			According to \cref{eq:weighted-matrix,eq:inner-tv,adjoint},
			\begin{equation}\nonumber
				\begin{array}{lcl}
					\Xi(\bm B)
					&=&-\langle\bar{\bm\lambda} ,\W_{\mathrm{c},T} \bm B\rangle
					+  \frac{1}{2N} \sum\limits_{(k,\f)\in\ftc}|\f|(\L_s\L_t^*\bm\lambda_0)_{k,\f} \|\bm B_{k,\f}\|^2
					\\[1mm]
					&=&\frac{1}{N}\sum\limits_{(k,\f)\in\ftc} |\f|\big(\frac{1}{2}( \L_{s}\L^*_{t}\bm\lambda_0)_{k,\f}\|\bm{B}_{k,\f}\|^2-\bar{\bm{\lambda}}_{k,\f}^{\top}\bm{B}_{k,\f}\big).
				\end{array}
			\end{equation}
			Note that $\bm\lambda_0\ge 0$ implies $\L_{s}\L^*_{t}\bm\lambda_0\ge 0$.
			Thus, the dual problem of \eqref{eq:opt-ddot} is given by
			\begin{equation}
				\label{eq:opt-ddot-dual}
				\begin{array}{ll}
					\max\limits_{(\bm\lambda_0,\bar{\bm{\lambda}})\in \mathbb{C}}
					\Big\{-\frac{1}{2N}\sum\limits_{(k,\f)\in \mathbb{I}}|\f|\frac{\|\bar{\bm{\lambda}}_{k,\f}\|^2}{(\L_{s}\L^*_{t}\bm\lambda_0)_{k,\f}}
					\, \ \big\vert \ \,
					\W_{\mathrm{c},V}\bm{c}+\A^*(\W_{\mathrm{st},V}\bm\lambda_0;\W_{\mathrm{c},T}\bar{\bm{\lambda}})=0
					\Big\},
				\end{array}
			\end{equation}
			where $\mathbb{I}:=\big\{(k,\f)\in\ftc\mid(\L_{s}\L^*_{t}\bm\lambda_0)_{k,\f}>0\big\}$ is an index set and $\mathbb{C}$ is the constraint set defined by
			$\mathbb{C}:=\big\{(\bm\lambda_0;\bar{\bm{\lambda}})
			\mid  \bm\lambda_0\ge 0, \ \mbox{and}\ \bar{\bm{\lambda}}_{k,\f} = 0 \text{ if }  (\L_{s}\L^*_{t}\bm\lambda_0)_{k,\f}=0\big\}$.

			We first show that \eqref{eq:opt-ddot-dual} admits a strictly feasible point.
			Define the vector $\bm{\xi}_0:=\frac{1}{N|V|}\W_{\mathrm{st},V}^{-1}\bm{1}_{|\fvs|}\in\mathbb{R}^{\fvs}$.
			According to \eqref{defA} one has
			\begin{equation}
				\label{eq:precxi}
				(\A_t^*\W_{\mathrm{st},V}\bm{\xi}_0)_{k,\v}=\begin{cases}
					-\frac{1}{|V|} &\mbox{if } k = 0,\\
					\frac{1}{|V|} & \mbox{if }  k = N,\\
					0              & \text{otherwise.}
				\end{cases}
			\end{equation}
			Then by 
			the definition of $\bm{c}$ in \eqref{eq:opt-ddot}, one has that
			\begin{equation}
				\label{eq:cxi}
				\sum_{\v\in V}(\W_{\mathrm{c},V}\bm{c}+\A_t^*\W_{\mathrm{st},V}\bm{\xi}_0)_{0,\v}
				= \sum_{\v\in V}(\W_{\mathrm{c},V}\bm{c}+\A_t^*\W_{\mathrm{st},V}\bm{\xi}_0)_{N,\v}=0.
			\end{equation}
			Fix a triangle $\f\in T$ with the vertices being denoted, without loss of generality, by $\v_1$, $\v_2$, and $\v_3$.
			Based on \eqref{eq:nabla-h} and \eqref{eq:As},
			we observe that if $\varphi_{k,\v_1}=\varphi_{k,\v_2}=\varphi_{k,\v_3}$ for a fixed time index $k$, then 
			\begin{equation*}\
				(\tilde \A_s\bm\varphi_k)_{\f} = \sum_{\v\in V_{\f}} \varphi_{k,\v}\nabla h_{\v}  = (J_{\f}^{\top})^\dag\begin{pmatrix}
					-1 &1 &0 \\
					-1 & 0& 1
				\end{pmatrix} \begin{pmatrix}
					\varphi_{k,\v_1}\\
					\varphi_{k,\v_2}\\
					\varphi_{k,\v_3}
				\end{pmatrix}=0.
			\end{equation*}
			On the other hand, if $(\tilde \A_s\bm{\varphi}_k)_{\f}=0$, one has from \eqref{eq:nabla-h} and \eqref{eq:As} that 
			\begin{equation*}
				0=J_{\f}^{\top}(\tilde \A_s\bm{\varphi}_k)_{\f}  = J_{\f}^{\top}(J_{\f}^{\top})^\dag\begin{pmatrix}
					-1 &1 &0 \\
					-1 & 0& 1
				\end{pmatrix} \begin{pmatrix}
					\varphi_{k,\v_1}\\
					\varphi_{k,\v_2}\\
					\varphi_{k,\v_3}
				\end{pmatrix}= \begin{pmatrix}
					-1 &1 &0 \\
					-1 & 0& 1
				\end{pmatrix} \begin{pmatrix}
					\varphi_{k,\v_1}\\
					\varphi_{k,\v_2}\\
					\varphi_{k,\v_3}
				\end{pmatrix}, 
			\end{equation*}
			implying that $\varphi_{k,\v_1}=\varphi_{k,\v_2}=\varphi_{k,\v_3}$.
			Then, it is easy to see that $\ker(\tilde \A_s)$ is spanned by $\bm{1}_{|V|}$.
			Moreover, since $\A_s = \I_{|\G^{\text{c}}_{\text{time}}|}\otimes\tilde \A_s$ by \eqref{defA}, $\ker(\A_s)$ is spanned by
			\begin{equation}
				\label{span}
				\big \{\e_1\otimes \bm{1}_{|V|},\ldots,\e_{N+1}\otimes \bm{1}_{|V|} \big \},
			\end{equation}
			where $\e_i\in\mathbb{R}^{N+1}$ are the standard unit vectors.
			Thus, from \eqref{eq:precxi}, \eqref{eq:cxi}, and the definition of $\bm{c}$ in \eqref{eq:opt-ddot}, we deduce that $\langle -\W_{\mathrm{c},V}\bm{c}-\A_t^*\W_{\mathrm{st},V}\bm{\xi}_0, \x\rangle=0$ for all $\x\in \ker(\A_s)$,
			which implies that $-\W_{\mathrm{c},V}\bm{c}-\A_t^*\W_{\mathrm{st},V}\bm{\xi}_0\in \rge(\A_s^*)$.
			Consequently, using the fact that $\W_{\mathrm{c},T}$ is nonsingular, there exists a vector $\bar{\bm{\xi}}\in \mathbb{R}^{3\times{\ftc}}$ such that
			\begin{equation}
				\label{feaseq}
				\A^*_s\W_{\mathrm{c},T}\bar{\bm{\xi}}=
				-\W_{\mathrm{c},V}\bm{c}-\A_t^*\W_{\mathrm{st},V}\bm{\xi}_0.
			\end{equation}
			Moreover, since $\bm\xi_0>0$, it follows from the definition of $\L_{t}^*$ and $\L_{s}$ that $\L_{s}\L_{t}^*\bm\xi_0>0$. This, together with \eqref{feaseq}, implies that $\bm\xi:=(\bm\xi_0;\bar{\bm\xi})$ is a strictly feasible point of \eqref{eq:opt-ddot-dual}.
			Note that the objective function of \eqref{eq:opt-ddot-dual} is non-positive.
			Then, by \cite[Lemma 36.1]{rock-1}, the objective function of \eqref{eq:opt-ddot} in the feasible set is bounded from below.
			
			Next, we show that \eqref{eq:opt-ddot} is also strictly feasible.
			Set $\hat{\bm\varphi}\in\mathbb{R}^{\fvc}$ with
			$\hat{\varphi}_{k,\v}:= -(k+1)$.
			Let
			$\hat{\q}:=\A\hat{\bm\varphi},$ and define $\hat{\bm{A}} :=  \A_t\hat{\bm\varphi}$ and $\hat{\bm{B}}:=\A_s\hat{\bm\varphi}$.
			Note that the components $\hat{\varphi}_{k,\v}$ depend solely on the index $k$ and are invariant with respect to the change of the index $\v$.
			It is easy to see that $\hat{\bm\varphi}$ lies in $\ker(\A_s)$, which is spanned by the vectors in \eqref{span}.
			So we have $\hat{\bm{B}}=0$.
			Moreover, since the value of $\hat{\varphi}_{k,\v}$ decreases as $k$ increases for any fixed $\v$,
			it follows that $\hat{A}_{k,\v}<0$ by \eqref{defA}.
			Thus, $(\hat{\bm\varphi},\hat{\q})$ is a strictly feasible point for problem \eqref{eq:opt-ddot}.
			
			Recall that the objective function of \eqref{eq:opt-ddot} in the feasible set is bounded from below.
			Using its strict feasibility and applying \cite[Theorem 28.2]{rock-1} implies that it admits a Kuhn-Tucker vector, which is the solution to the dual problem \eqref{eq:opt-ddot-dual} by  \cite[Theorems 30.4 \& 30.5]{rock-1}.
			Note that the optimal value of \eqref{eq:opt-ddot} is finite and \eqref{eq:opt-ddot-dual} was shown to be strictly feasible,
			it also admits a Kuhn-Tucker vector by \cite[Theorem 28.2]{rock-1}.
			Then we know from \cite[Corollary 28.3.1]{rock-1} that the solution set to the KKT system \eqref{eq:ddot-kkt} is non-empty.
			This completes the proof.
		\end{proof}

		\section{Decoupling and SOCP reformulation}
		\label{sec:soc}
		This section reformulates the discretized dual DOT problem \eqref{eq:opt-ddot} as a linear SOCP to ease the algorithmic design in the next section.
			Our first step is to decouple the quadratic inequalities in the constraint set $\mathbb{P}$ defined by \eqref{eq:dsP}, since $\L_t\L_s^*(\vertiii{\bm B}^2)$ couples the triangle-indexed spatial variable $\bm B$ with the vertex-indexed inequalities, and then to transfer the decoupled quadratic inequality constraints to conic constraints. 
			
			Let $(\mathbb{R}^{3\times{\fts}})^6$ denote the Cartesian product of six copies of $\mathbb{R}^{3\times{\fts}}$, and write $\y \in(\mathbb{R}^{3\times{\fts}})^6$ as $\y = (\y_1;\ldots;\y_6)$ with each $\y_i\in \mathbb{R}^{3\times{\fts}}$.
		Define the linear operators
		$\widetilde{\F}\colon \mathbb{R}^{3\times{\ftc}}\rightarrow (\mathbb{R}^{3\times{\fts}})^6$
		and $\F\colon \mathbb{R}^{\fvs} \times \mathbb{R}^{3\times{\ftc}} \to \mathbb{R}^{\fvs}\times (\mathbb{R}^{3\times{\fts}})^6$ by
		\begin{equation}
			\label{eq:tildef}
			\widetilde{\F}: =
			\begin{pmatrix}
				\bm{1}_3\otimes \mathcal{C}_-\otimes \I_{3|T|}
				\\
				\bm{1}_3\otimes \mathcal{C}_+\otimes \I_{3|T|}
			\end{pmatrix}
			\quad\mbox{and}\quad
			\mathcal{F} :=\diag\big(\I_{|\fvs|},\widetilde{\mathcal{F}}\big),
		\end{equation}
		for the purpose of copying vectors from centered time grids to staggered ones,  
		where  $\mathcal{C}_-,\mathcal{C}_+\colon \mathbb{R}^{|\G_{\text{time}}^{\text{c}}|}\rightarrow \mathbb{R}^{|\G_{\text{time}}^{\text{st}}|}$ are defined by
		\begin{equation}
			\label{eq:ftp}
			\mathcal{C}_- := \begin{pmatrix}
				\I_{|\G_{\text{time}}^{\text{st}}|} & \bm{0}_{|\G_{\text{time}}^{\text{st}}|}
			\end{pmatrix}\quad\mbox{and}\quad
			\mathcal{C}_+ := \begin{pmatrix}
				\bm{0}_{|\G_{\text{time}}^{\text{st}}|}& \I_{|\G_{\text{time}}^{\text{st}}|}
			\end{pmatrix}.
		\end{equation}
		Then, for $\q \equiv (\bm{A};\bm{B})\in \mathbb{R}^{\fvs} \times \mathbb{R}^{3\times{\ftc}}$, one can write
			$$
			\F\q = (\bm{A};\bm{D}^-;\bm{D}^+) \quad\mbox{and} \quad \bm{D}^\pm = (\bm{D}^\pm_1;\bm{D}^\pm_2;\bm{D}^\pm_3),
			$$
			where each $\bm{D}^\pm_i\in\mathbb{R}^{3\times{\fts}}$. More explicitly, for each $(k,\f) \in \fts$ and $i = 1,2,3,$ one has $(\bm D_i^-)_{k,\f} = \bm B_{k-\frac{1}{2},\f}$ and $(\bm D_i^+)_{k,\f} = \bm B_{k+\frac{1}{2},\f}$. 
			
			For each triangle $\f \in T$, fix an ordering of its three vertices and denote them by $\{\v_1^{\f},\v_2^{\f},\v_3^{\f}\}$.
		For $\v\in V$, define the set of triangles in $T_{\v}$ for which $\v$ is identified as the $i$-th vertex of triangle $\f$ by
		\begin{equation}
			\label{eq:set-tvi}
			T_{\v,i} := \{\f\in T_{\v}\mid \v^{\f}_i = \v \}, \quad i=1,2,3.
		\end{equation}
		Then, for every $\v\in V$, the sets $T_{\v,1}, T_{\v,2},T_{\v,3}$ form a disjoint partition of $T_{\v}$. Moreover, for each fixed $i=1,2,3,$ the sets $\{T_{\v,i}\}_{\v\in V}$ form a disjoint partition of $T$, i.e.,
		\begin{equation}
			\label{propertytvi}
			T_{\v,i}\cap T_{\v,j}=\emptyset\
			\forall\, i\neq j,
			\quad
			\bigcup\limits_{i=1,2,3}
			T_{\v,i}=T_{\v},
			\quad
			\mbox{and}
			\quad
			\bigcup\limits_{\v\in V}T_{\v,i}=T
			\,\ \forall\, i=1,2,3.
		\end{equation}
		Define the diagonal linear operators $\D_i:\mathbb{R}^{3\times{\fts}}\to
		\mathbb{R}^{3\times{\fts}}$ by
		\begin{equation}
			\label{defD}
			\begin{array}{ll}
				\D_i:=
				\I_{|\G^{\text{st}}_{\text{time}}|}
				\otimes
				\diag
				\Big(\sqrt{\frac{|\f_1|}{|\v_i^{\f_1}|}},\ldots,
				\sqrt{\frac{|\f_{|T|}|}{|\v_i^{\f_{|T|}}|}}
				\,\Big)
				\otimes \I_3,
				\quad
				i=1,2,3.
			\end{array}
		\end{equation}
		For $(k,\v) \in \fvs$, and $i=1,2,3,$ define the linear operators
			$\P_{i}^{k,\v}\colon
			\mathbb{R}^{3\times{\fts}}\to \mathbb{R}^{3|T_{\v,i}|}$ by
		\begin{equation}
			\label{eq:pvi}
			\begin{array}{ll}
				\P_{i}^{k,\v}\bm{D}:=\Big(
				\bm{D}_{k,\f_1^{\v,i}};\ldots; \bm{D}_{k,\f_{|T_{\v,i}|}^{\v,i}}\Big),
				\quad
				\bm{D}\in \mathbb{R}^{3\times{\fts}},
			\end{array}
		\end{equation}
		where $\f_1^{\v,i},\ldots,\f^{\v,i}_{|T_{\v,i}|}$ are the ordered triangles in  $T_{\v,i}$.
		Thus, $\P_{i}^{k,\v}$ selects the components associated with the specific time index $k$ and the triangles in $T_{\v,i}$.
		Then, we define the linear operators
		$\T_{i}^{k,\v}\colon
			\mathbb{R}^{3\times{\fts}}\to \mathbb{R}^{3|T_{\v,i}|}$ ($(k,\v)\in\fvs$, and $i=1,2,3$) by
		\begin{equation}
			\label{eq:tvkvi}
			\begin{array}{rl}
				\T_{i}^{k,\v}\bm{D}
				&=\P_{i}^{k,\v}\D_i\bm{D}
				=\Big(\sqrt{\frac{|\f_1^{\v,i}|}{|\v|}}\bm{D}_{k,\f_1^{\v,i}};\ldots;\sqrt{\frac{|\f_{|T_{\v,i}|}^{\v,i}|}{|\v|}}\bm{D}_{k,\f_{|T_{\v,i}|}^{\v,i}}\Big),
				\quad
				\bm{D}\in \mathbb{R}^{3\times{\fts}}.
			\end{array}
		\end{equation}
		Note that, if $\f\in T_{\v,i},$ then $\v^{\f}_i = \v$, and hence $|\v| = |\v_i^{\f}|$.
				
		Using the above definitions, we decouple the inequality constraints of $\mathbb{P}$ in \eqref{eq:dsP}.
		Specifically, for $(\bm{A};\bm{B}) \in \mathbb{R}^{\fvs}\times\mathbb{R}^{3\times{\ftc}}$, it follows from \eqref{eq:lxadj} that
		\begin{equation*}
			\begin{array}{ll}
				(\mathcal{L}_{s}^*(\vertiii{\bm{B}}^2))_{k,\v} = \frac{1}{3|\v|}\sum\limits_{\f\in T_{\v}} |\f| \|{\bm{B}}_{k,\f}\|^2
				\quad \forall\,
				(k,\v)\in \fvc.
			\end{array}
		\end{equation*}
		By \eqref{eq:li}, for every $(k,\v)\in \fvs$, 
		\begin{equation*}
			\begin{array}{lll}
				\big(\L_{t}(\mathcal{L}_{s}^*(\vertiii{\bm{B}}^2))\big)_{k,\v}
				&=&\frac{1}{2}\big(
				(\mathcal{L}_{s}^*(\vertiii{\bm{B}}^2))_{k+\frac{1}{2},\v}+(\mathcal{L}_{s}^*(\vertiii{\bm{B}}^2))_{k-\frac{1}{2},\v}
				\big)
				\\[2mm]
				&=& \frac{1}{6|\v|} \big(\sum\limits_{\f\in T_{\v}}|\f|\|\bm B_{k-\frac{1}{2},\f}\|^2+\sum\limits_{\f\in T_{\v}}|\f|\|\bm B_{k+\frac{1}{2},\f}\|^2\big).
			\end{array}
		\end{equation*}
		Set $(\bm{D}^-;\bm{D}^+):=\widetilde \F \bm{B}$ with $\cal \widetilde F$ being defined by \eqref{eq:tildef}. Then, by
		\eqref{eq:set-tvi}, \eqref{propertytvi}, and \eqref{eq:tvkvi}, for every $(k,\v)\in \fvs$,
		\begin{equation}
			\label{eq:tildep}
			\begin{array}{ll}
				A_{k,\v} + \frac{1}{2} (\L_{t}(\mathcal{L}_{s}^*\vertiii{\bm{B}}^2))_{k,\v}
				=A_{k,\v} + \frac{1}{12}\sum\limits_{i=1}^3(\|\T_{i}^{k,\v}\bm{D}^-_i\|^2+\|\T_{i}^{k,\v}\bm{D}^+_i\|^2),
			\end{array}
		\end{equation}
		which decouples the inequality constraint in $\mathbb{P}$.
		
		We next express each decoupled inequality as an equivalent second-order cone constraint.
		Note that the left-hand side of \eqref{eq:tildep} further equals to 
		\begin{equation}
			\label{eq:tildep2}
			\begin{array}{ll}
				\frac{1}{4}\big((1+A_{k,\v})^2 + \frac{1}{3}\sum\limits_{i=1}^3(\|\T_{i}^{k,\v}\bm{D}^-_i\|^2+\|\T_{i}^{k,\v}\bm{D}^+_i\|^2)- (1-A_{k,\v})^2\big).
			\end{array}
	\end{equation}
	Define the linear operator 
		$\B\colon \mathbb R^{\fvs}\times(\mathbb{R}^{3\times{\fts}})^6
		\to
		\mathbb R^{\fvs}\times(\mathbb{R}^{3\times{\fts}})^6
		\times\mathbb R^{\fvs}$ by 
		\begin{equation*}
			\begin{array}{ll}
				\B (\x;\y) := \big( -\x;\frac{\sqrt{3}}{3}\y; \x\big),
				\quad 
				\x \in\mathbb R^{\fvs}\ , \y\in (\mathbb{R}^{3\times{\fts}})^6.
			\end{array}
		\end{equation*} 
	Thus, by $\F\q \equiv \F (\bm A;\bm{B})
		=(\bm A; \bm D^-; \bm D^+)$, it follows that 
		\begin{equation*}
			\begin{array}{l}
				\B\F\q + \d = \big(\bm 1_{|\fvs|} - \bm A; \frac{\sqrt{3}}{3}\bm D^-; \frac{\sqrt{3}}{3}\bm D^+;\bm 1_{|\fvs|} + \bm A\big),
			\end{array}
		\end{equation*}
		where $\d: = (\bm{1}_{|\fvs|};\bm{0}_{18|\fts|};\bm{1}_{|\fvs|})$.
		For each $(k,\v)\in\fvs$, define
		the linear operators $\T^{k,\v}\colon \mathbb R^{\fvs}\times(\mathbb{R}^{3\times{\fts}})^6
		\times\mathbb R^{\fvs}
		\to
		\mathbb {R}^{2+6|T_{\v}|}$ by
	\begin{equation}
		\label{eq:tvkv}
\T^{k,\v}:=\diag\big(\e_{k,\v}^{\top},\T_{1}^{k,\v},\T_{2}^{k,\v},\T_{3}^{k,\v},\T_{1}^{k,\v},\T_{2}^{k,\v},\T_{3}^{k,\v},\e_{k,\v}^{\top}\big),
	\end{equation}
	where $\bm e_{k,v}\in\mathbb{R}^{\fvs}$ is the unit vector associated with $(k,\v)$, and the operators $\T_{i}^{k,\v}$ (for $i=1,2,3$) are defined by  \eqref{eq:tvkvi}.
	Then, from \eqref{eq:tildep}, \eqref{eq:tildep2}, and the definitions above, one can see that
	\begin{equation}
		\label{eqsoc}
		\begin{array}{ll}
			A_{k,\v} + \frac{1}{2} \big(\L_{t}(\mathcal{L}_{s}^*\vertiii{\bm{B}}^2)\big)_{k,\v}\le 0
			\quad\Leftrightarrow\quad
			\T^{k,\v}(\B\F\q+\d) \in\mathbb{K}^{k,\v}_{\mathrm{soc}}
			\quad \forall\, (k,\v)\in \fvs,
		\end{array}
	\end{equation}
	where the dimension of each second-order cone $\mathbb{K}^{k,\v}_{\mathrm{soc}}$ is $2+6|T_{\v}|$.
	\begin{remark}
		\label{rmkbdry}
		It is easy to see that the strict inequality (resp., equality) in \eqref{eqsoc} corresponds to the interior (resp., boundary) of  $\mathbb{K}^{k,\v}_{\mathrm{soc}}$.
	\end{remark}

	 By stacking the constraints in \eqref{eqsoc} over all $(k,\v)\in \fvs$, we obtain the product cone used in the SOCP reformulation. Recall that $V = \{\v_1,\ldots,\v_{|V|}\}$. Define
	\begin{equation}
		\label{eq:setQ}
		\mathbb{Q} := \mathbb{Q}^{\frac{1}{2}}\times\mathbb{Q}^{\frac{3}{2}}\times\cdots\times\mathbb{Q}^{N-\frac{1}{2}} \quad \text{with}
		\quad \mathbb{Q}^{k}: = \mathbb{K}^{k,\v_1}_{\mathrm{soc}}\times\cdots\times\mathbb{K}^{k,\v_{|V|}}_{\mathrm{soc}}.
	\end{equation}
	Introducing the auxiliary variable $\z:=\B\F\q+\d$, we equivalently reformulate \eqref{eq:opt-ddot} as the following SOCP problem
	\begin{equation}
		\label{eq:opt-ddot-soc}
		\min_{\bm{\varphi},\q,\z}\big\{ \langle \bm{\varphi},\bm{c}\rangle_{\mathrm{c},V} \ \big\vert\ \mathcal{A}\bm{\varphi} = \q,\, \B\F\q+\d =\z,\, \T\z\in\mathbb{Q}
		\big\},
	\end{equation}
	where the linear operator $\T$ is defined by
	\begin{equation}
		\label{eq:tvQ}
		\T\z:=\big(
		\T^{\frac{1}{2},\v_1}\z;\ldots;\T^{\frac{1}{2},\v_{|V|}}\z;\T^{\frac{3}{2},\v_1}\z;\ldots;\T^{\frac{3}{2},\v_{|V|}}\z;\ldots;
		\T^{N-\frac{1}{2},\v_{|V|}}\z\big).
	\end{equation}

		\begin{remark}
			The major difference between \eqref{eq:opt-ddot-soc} and the SOCP problem \cite[Eq. (3.14)]{dotsoc} for DOT in flat domains is that the conic constraint here is imposed on $\T\bm z$, instead of on the auxiliary variable $\bm z$ directly.
			The underlying reason is that the simple shifted copy operators are sufficient to decouple the discretized dual DOT problem in \cite{dotsoc} since the staggered grids are sufficiently regular. Yet, the triangular mesh for DOT on geometric surfaces does not admit such regularity.  
			Moreover, the presence of $\T$ prevents a direct application of the KKT analysis and algorithmic implementation in \cite{dotsoc}, which rely on a cone constraint imposed directly on the auxiliary variable $\bm z$. Thus, the structure of \eqref{eq:opt-ddot-soc} should be examined with further theoretical effort.
	\end{remark}
	
	\subsection{Analysis of SOCP reformulation}
	Here, we conduct a theoretical analysis of the SOCP \eqref{eq:opt-ddot-soc} to reveal essential properties crucial for designing and implementing algorithms, particularly regarding efficiency.
	The first result is about the linear operators in the reformulation.
	
	\begin{proposition}
		\label{prop:diag}
		For the linear operators $\F$ and $\B$ in the SOCP problem \eqref{eq:opt-ddot-soc}, the composite operator $\F^*\B^*\B\F$ is diagonal.
	\end{proposition}
	\begin{proof}
		According to the definitions of $\mathcal{C}_-$ and $\mathcal{C}_+$ in \eqref{eq:ftp}, we obtain that $\mathcal{C}_{-}^*\mathcal{C}_{-}+\mathcal{C}_{+}^*\mathcal{C}_{+} = \diag(1,2,\ldots,2,1)$.
		Then, using the basic properties of the Kronecker product \cite[(1.3.1)-(1.3.4)]{golub}, one has from \eqref{eq:tildef} that
		\begin{equation}\nonumber
			\begin{array}{ll}
				\widetilde{\F}^*\widetilde{\F} =  \bm{1}_3^{\top}\bm{1}_3\otimes(\mathcal{C}_{-}^*\mathcal{C}_{-}+\mathcal{C}_{+}^*\mathcal{C}_{+})\otimes \I_{3|T|}^*\I_{3|T|} =3\diag(1,2,\ldots,2,1)
				\otimes\I_{3|T|},
			\end{array}
		\end{equation}
		which is diagonal. Then, by incorporating the definition of $\B$, one gets
		\begin{equation*}
			\begin{array}{ll}
				\F^*\B^*\B\F
				=
				\F^* \diag\left(2\I_{|\fvs|} ,
					\frac{1}{3} \I_{18|\fts|}\right)
				\F = \diag\big(
				2\I_{|\fvs|},
				\frac{1}{3}\widetilde{\F}^*\widetilde{\F}
				\big),
			\end{array}
		\end{equation*}
		which is also a diagonal matrix, and this completes the proof.
	\end{proof}

	\begin{proposition}
		\label{propk:inverse}
		The linear operator $\T$ given by \eqref{eq:tvQ} satisfies
		$\T=\P\D$ and $\T^{-1}=\D^{-1}\P^*$ with $\D$ being diagonal and $\P$ being a permutation matrix, defined by 
		\begin{equation}
			\label{defpd}
			\begin{cases}
				\D:=
				\diag\left(\I_{|\fvs|},\
				\I_2
				\otimes \diag(\D_1,\D_2,\D_3),\
				\I_{|\fvs|}\right),
				\\[1mm]
				\P:=\big(
				\P^{\frac{1}{2},\v_1};\ldots;\P^{\frac{1}{2},\v_{|V|}};
				\P^{\frac{3}{2},\v_1};\ldots;\P^{\frac{3}{2},\v_{|V|}};
				\ldots;
				\P^{N-\frac{1}{2},\v_{|V|}}\big),
			\end{cases}
		\end{equation}
		where  $\P^{k,\v}:=\diag\big(\e_{k,\v}^{\top},\P_{1}^{k,\v},\P_{2}^{k,\v},\P_{3}^{k,\v},\P_{1}^{k,\v},\P_{2}^{k,\v},\P_{3}^{k,\v},\e_{k,\v}^{\top}\big)$.
	\end{proposition}
	
	\begin{proof}
			Given $\z\in \mathbb{R}^{\fvs} \times (\mathbb{R}^{3\times{\fts}})^6 \times\mathbb{R}^{\fvs}$,
			by the definitions of $\T_i^{k,\v}$ and $\P_i^{k,\v}$ in \eqref{eq:pvi} and \eqref{eq:tvkvi}, for $i = 1,2,3,$ one has $\T_i^{k,\v} = \P^{k,\v}_i \D_i $.
			Therefore, by \eqref{eq:tvkv}, for each $(k,\v)\in\fvs$, we obtain $\T^{k,\v}=\P^{k,\v}\D$. Then by \eqref{eq:tvQ}, we know that $\T = \P\D$.
			
			For each $(k,\v)\in \fvs$, the first and last diagonal blocks of $\P^{k,\v}$ are both $\e_{k,\v}^{\top}$. For the six middle blocks, fix $i = 1,2,3$. 
			By \eqref{propertytvi}, the sets $(T_{\v,i})_{\v\in V}$ form a disjoint partition of $T$. 
			Hence, by \eqref{eq:pvi}, for each fixed $k\in \G^{\rm st}_{\rm time}$, the operators $\P^{k,\v}_{i}$, $\v\in V$, selects each triangle component $\bm D_{k,\f}\in\mathbb{R}^3$ with $\f\in T$ exactly once. 
			It follows that every row and every column of $\P$ has exactly one nonzero entry, equal to one. Therefore, $\P$ is a permutation matrix.
		Note that $\D$ is nonsingular, so that $\T^{-1}=\D^{-1}\P^*$, which completes the proof.
	\end{proof}
	
	Next, we focus on the KKT system of the SOCP reformulation \eqref{eq:opt-ddot-soc} and establish its relationship to \eqref{eq:ddot-kkt}.
	Denote the dual variable corresponding to the equality constraint $\A\bm{\varphi}= \q$ by $\bm{\alpha}:=(\bm{\alpha}_1;\bm{\alpha}_2)\in\mathbb{R}^{\fvs}\times \mathbb{R}^{3\times{\ftc}}$ (where $\q \equiv (\bm A; \bm B)$). The dual variable associated with the equality constraint $\z = \B\F\q+\d$ is denoted by $\bm{\beta}:=(\bm{\beta}_{1};\bm{\beta}_{2};\bm{\beta}_{3})\in\mathbb{R}^{\fvs}\times (\mathbb{R}^{3\times{\fts}})^6 \times \mathbb{R}^{\fvs}$,
	where $\bm\beta_2 := (\bm\beta_{2,1};\ldots ;\bm\beta_{2,6})\in(\mathbb{R}^{3\times{\fts}})^6$
	with $\bm\beta_{2,i}\in \mathbb{R}^{3\times{\fts}}$ for $i=1,\ldots,6$.
	Similarly, we write $\z=(\z_{1};\z_{2};\z_{3})$ with $\z_2=(\z_{2,1};\ldots ;\z_{2,6})$.
	The Lagrangian function for \eqref{eq:opt-ddot-soc} is then given by
	\begin{equation*}
		\mathfrak{L}(\bm{\varphi},\z,\q;\bm{\alpha},\bm{\beta}) := \langle \bm\varphi,\bm{c}\rangle_{\mathrm{c},V} + \delta_{\mathbb{Q}}(\T\z)+ \langle \bm{\alpha},\mathcal{A}\bm{\varphi}-\q\rangle_{q} +\langle \bm{\beta},\z-\mathcal{B}\mathcal{F}\q-\d\rangle_z,
	\end{equation*}
	where, for
		$\bm{\alpha},\bm{\alpha}'\in \mathbb{R}^{\fvs}\times \mathbb{R}^{3\times{\ftc}}$
		and
		$\bm{\beta},\bm{\beta}'\in \mathbb{R}^{\fvs}\times(\mathbb{R}^{3\times{\fts}})^6\times\mathbb{R}^{\fvs}$, the inner products are defined by
	\begin{equation}
		\label{eq:innerqz}
		\begin{cases}
			\langle \bm{\alpha},\bm{\alpha}'\rangle_{q}
			:= \langle {\bm{\alpha}}_1, {\bm{\alpha}}'_1\rangle_{\mathrm{st},V}
			+ \langle {\bm{\alpha}}_2,  {\bm{\alpha}}'_2\rangle_{\mathrm{c},T},
			\\
			\langle {\bm{\beta}},{\bm{\beta}'}\rangle_z:
			=\langle  {\bm{\beta}}_{1},  {\bm{\beta}}'_{1}\rangle_{\mathrm{st},V}
			+\sum_{i=1}^6\langle {\bm{\beta}}_{2,i},{\bm{\beta}}'_{2,i}\rangle_{\mathrm{st},T}
			+\langle {\bm{\beta}}_{3}, {\bm{\beta}}'_{3}\rangle_{\mathrm{st},V},
		\end{cases}
	\end{equation}
	where the inner products on the right-hand sides are defined in \eqref{eq:inner-tv}.
	Based on the Lagrangian function above, the KKT system of \eqref{eq:opt-ddot-soc} is given by
	\begin{equation}
		\label{eq:ddot-kkt-soc}
		\begin{cases}
			\A\bm{\varphi}=\q,\ \B\F\q+\d=\z, \\
			\W_{\mathrm{c},V}\bm{c}+\A^*(\W_{\mathrm{st},V}\bm\alpha_1;\W_{\mathrm{c},T}\bm\alpha_2)=0,\\
			(\W_{\mathrm{st},V}\bm\alpha_1;\W_{\mathrm{c},T}\bm\alpha_2)+\F^*\B^*\diag
			(\W_{\mathrm{st},V},\I_6\otimes \W_{\mathrm{st},T},\W_{\mathrm{st},V})\bm{\beta}=0,\\
			\T\z\in \mathbb{Q},\
			\diag (\W_{\mathrm{st},V},\I_6\otimes \W_{\mathrm{st},T},\W_{\mathrm{st},V})\bm{\beta}\in\T^* \mathbb{Q},\
			\langle \bm{\beta},\z\rangle_z=0,
		\end{cases}
	\end{equation}
	where the weight matrices $\W_{\mathrm{c},V},\W_{\mathrm{st},V},\W_{\mathrm{c},T},\W_{\mathrm{st},T}$ are defined by \eqref{eq:weighted-matrix}, and $\T^*$ is the Euclidean adjoint of $\T$.
	The following result ensures that the solution set for the KKT system \eqref{eq:ddot-kkt-soc} is not empty. Moreover, it guarantees that one can recover a solution to the KKT system \eqref{eq:ddot-kkt} from a solution of the KKT system \eqref{eq:ddot-kkt-soc}.
	\begin{theorem}
		\label{prop:soc-kkt}
		The solution set to the KKT system \eqref{eq:ddot-kkt-soc} is not empty.
		Moreover, for any solution $(\bm{\varphi};\q;\z;\bm{\alpha};\bm{\beta})$ to \eqref{eq:ddot-kkt-soc}, the subvector $(\bm{\varphi};\q;\bm{\alpha}_1;\bm{\alpha}_2;-\bm\alpha_1)$ is a solution to the KKT system \eqref{eq:ddot-kkt}.
	\end{theorem}

	\begin{proof}
		Since \eqref{eq:opt-ddot-soc} is an equivalent reformulation of \eqref{eq:opt-ddot}, the two problems have the same optimal value.
		According to \Cref{prop:dotkkt}, the optimal value of \eqref{eq:opt-ddot-soc} is both finite and achievable.
		Moreover, from \eqref{eq:tildep} and \eqref{eqsoc} one can see that any strictly feasible point $(\bm\varphi;\q)$ for \eqref{eq:opt-ddot} can be extended to a strictly feasible point $(\bm\varphi;\q;\z)$ for \eqref{eq:opt-ddot-soc} with $\z: = \B\F\q+\d$.
		Thus, by \cite[Theorem 28.2]{rock-1}, the KKT system \eqref{eq:ddot-kkt-soc} has a nonempty solution set.
		In the following, we fix $(\bm{\varphi};\q;\z;\bm{\alpha};\bm{\beta})$ as a solution to \eqref{eq:ddot-kkt-soc} and define $\bm{\lambda}_0:=\bm{\alpha}_1$, $\bm{\bar\lambda}:=\bm{\alpha}_2$, and $\bm{\lambda}_q:=-\bm{\alpha}_1$.
		
		It is easy to see that the first two lines of \eqref{eq:ddot-kkt} hold according to the first two lines of \eqref{eq:ddot-kkt-soc}.
		Moreover, since $\T$ is invertible by \Cref{propk:inverse}, one has $\u:=(\T^*)^{-1}\diag
		(\W_{\mathrm{st},V},\I_6\otimes \W_{\mathrm{st},T},\W_{\mathrm{st},V})\bm{\beta}
		\in\mathbb{Q} $ and
		$\langle \u,\T\z\rangle=\langle\T^*\u,\z\rangle=\langle  \bm{\beta},\z\rangle_{z} = 0$.
		Consequently, it holds that
		\begin{equation}
			\label{eq:comple}
			\mathbb{K}_{\mathrm{soc}}^{k,\v}\ni \u^{k,\v} \perp \T^{k,\v}\z\in\mathbb{K}_{\mathrm{soc}}^{k,\v}
			\qquad\forall \, (k,\v)\in \fvs.
		\end{equation}
		Therefore, by \eqref{eqsoc} one has $\bm A+ \frac{1}{2} (\L_{t}(\mathcal{L}_{s}^*(\vertiii{\bm{B}}^2)))\le 0$.
		Fix $(k,\v) \in \fvs$.
		The relation $\z=\B\F\q+\d$ implies $(\z_{1})_{k,\v} = 1-A_{k,\v}$ and $(\z_{3})_{k,\v} = 1+A_{k,\v}$, which ensures that the vector $\T^{k,\v}\z$ is always nonzero by \eqref{eq:tvkv}.
		Let $u^{k,\v}_{0}$ be the first element of $\u^{k,\v}$, ${{\overline\u}^{k,\v}}$ be the vector consisting of the remaining elements of $\u^{k,\v}$, and  $u^{k,\v}_{1+6|T_{\v}|}$ be the last element of $\u^{k,\v}$.
		By \Cref{propk:inverse} and the fact that $\P$ is a permutation matrix, $(\T^*)^{-1}=(\D^*\P^*)^{-1}=\P\D^{-1}$.
		Then by \eqref{eq:weighted-matrix} and \eqref{defpd} one can see that
		\begin{equation}
			\label{equbeta}
			\begin{array}{ll}
				u^{k,\v}_0 = (\W_{\mathrm{st},V}\bm\beta_1)_{k,\v}=\frac{|\v|}{N}(\bm{\beta_1})_{k,\v}
				\quad\mbox{and}\quad
				u^{k,\v}_{1+6|T_{\v}|} = (\W_{\mathrm{st},V}\bm\beta_3)_{k,\v}
				=\frac{|\v|}{N}(\bm{\beta_3})_{k,\v}.
			\end{array}
		\end{equation}
		Since $\u^{k,\v}\in \mathbb{K}^{k,\v}_{\text{soc}}$, according to the third line of \eqref{eq:ddot-kkt-soc} and the definitions of $\F$ and $\B$, one has
		$(\bm{\alpha}_1)_{k,\v} = (\bm{\beta}_{1}-\bm{\beta}_{3})_{k,\v}\ge 0$, so that $\bm{\lambda}_0\ge 0$.
		In the following, we validate the complementarity condition in \eqref{eq:ddot-kkt}.
		Note that if $(\bm\alpha_1)_{k,\v}>0$, one has $\u^{k,\v}\neq 0$ by \eqref{equbeta}.
		Then by \eqref{eq:comple} we know that both $\u^{k,\v}$ and $\T^{k,\v}\z$ lie on the boundary of the second-order cone $\mathbb{K}^{k,\v}_{\mathrm{soc}}$.
		Thus, we know from \eqref{eqsoc} and \Cref{rmkbdry} that
		$A_{k,\v} + \frac{1}{2} (\L_{t}(\mathcal{L}_{s}^*\vertiii{\bm{B}}^2))_{k,\v} =0$.
		Consequently, the third line of \eqref{eq:ddot-kkt} also holds.

		Recall that $\T^{k,\v}\z$ is nonzero.
		If  $\u^{k,\v}$ is also nonzero (when $(\bm{\alpha}_1)_{k,\v}>0$ or
		$(\bm{\alpha}_1)_{k,\v} = 0$ but $(\bm{\beta}_{1})_{k,\v}=(\bm{\beta}_{3})_{k,\v}\neq0$), by \eqref{eq:comple} and \cite[Lemma 15]{soc} we know that \begin{equation}
			\label{eq:soc-z-beta}
			\T^{k,\v}\z = \kappa_{k,\v} \big(u^{k,\v}_{0};-\overline{\u}^{k,\v}\big)
			\quad \mbox{for a constant } \kappa_{k,\v}>0,
		\end{equation}
		which implies by \eqref{eq:tvkv} that
		\begin{equation}
			\label{eq:zfst-zlst}
			\begin{cases}
				1-A_{k,\v} =  (\z_{1})_{k,\v} =\kappa_{k,\v}u^{k,\v}_{0} =\kappa_{k,\v}(\W_{\mathrm{st},V}\bm{\beta}_{1})_{k,\v}, \\[0.2em]
				1+A_{k,\v} =(\z_{3})_{k,\v} =-\kappa_{k,\v}u^{k,\v}_{1+6|T_{\v}|}= -\kappa_{k,\v}(\W_{\mathrm{st},V}\bm{\beta}_{3})_{k,\v}.
			\end{cases}
		\end{equation}
		Summing the two equalities in \eqref{eq:zfst-zlst} implies that
		\begin{equation}
			\begin{array}{ll}
				\label{eq:alpha-eta}
				(\bm{\alpha}_1)_{k,\v} = (\bm{\beta}_{1}-\bm{\beta}_{3})_{k,\v} = \frac{2N}{|\v|\kappa_{k,\v}}.
			\end{array}
		\end{equation}
		Therefore, $(\bm{\alpha}_1)_{k,\v}=0$ implies that $(\bm{\beta}_{1})_{k,\v}=(\bm{\beta}_{3})_{k,\v}= 0$, hence $\bm{u}^{k,\v}=0$ by \eqref{equbeta}.
		Thus, \eqref{eq:soc-z-beta} holds only for the case that $(\bm{\alpha}_1)_{k,\v} > 0$.
		As a result, we obtain that
		\begin{equation}
			\label{relationutz}
			\big(u^{k,\v}_{0};-\overline{\u}^{k,\v}\big)=\tilde\kappa_{k,\v} \T^{k,\v}\z
			\quad\mbox{with}\quad
			\tilde{\kappa}_{k,\v}\eqdef \begin{cases}
				1/\kappa_{k,\v} &  \text{if } (\bm{\alpha}_1)_{k,\v} >0,\\
				0           & \text{if } (\bm{\alpha}_1)_{k,\v} = 0,
			\end{cases}
		\end{equation}
		where $\kappa_{k,\v}>0$ is given by \eqref{eq:soc-z-beta}.
		Taking $\u=(\T^*)^{-1}\diag
		(\W_{\mathrm{st},V},\I_6\otimes \W_{\mathrm{st},T},\W_{\mathrm{st},V})\bm{\beta}$
		and $\z=\B\F\q+\d$ into \eqref{relationutz} and using \Cref{propk:inverse} yields, for $j=1,2,3,$
		\begin{equation*}
			\begin{cases}
				(\W_{\mathrm{st},T}\bm{\beta}_{2,j})_{k,\f} = -\frac{|\f|}{|\v_j^{\f}|}\tilde{\kappa}_{k,\v_j^{\f}}(\z_{2,j})_{k,\f}=-\frac{\sqrt{3}|\f|}{3|\v_j^{\f}|}\tilde{\kappa}_{k,\v_j^{\f}}\bm{B}_{k-\frac{1}{2},\f},  
				\\
				(\W_{\mathrm{st},T}\bm{\beta}_{2,j+3})_{k,\f} = -\frac{|\f|}{|\v_j^{\f}|}\tilde{\kappa}_{k,\v_j^{\f}}(\z_{2,j+3})_{k,\f}=-\frac{\sqrt{3}|\f|}{3|\v_j^{\f}|}\tilde{\kappa}_{k,\v_j^{\f}}\bm{B}_{k+\frac{1}{2},\f},
			\end{cases}
		\end{equation*}
		where $\v_1^{\f},\v_2^{\f}$, and $\v_3^{\f}$ are vertices of the given triangle $\f\in T$.
		Note that the third line of \eqref{eq:ddot-kkt-soc} implies
		$\bm{\alpha}_2 = -\frac{\sqrt{3}}{3}\W_{\mathrm{c},T}^{-1}\widetilde{\F}^*(\W_{\mathrm{st},T}\bm{\beta}_{2,1};\ldots;\W_{\mathrm{st},T}\bm{\beta}_{2,6})$.
		Then by using the definition of $ \widetilde{\F}$ in \eqref{eq:tildef}, for $(k,\f)\in \ftc$, one can get from the above equality that
		\begin{equation*}
			(\bm{\alpha}_2)_{k,\f} =
			\begin{cases}
				\Big(\sum\limits_{j=1}^3 \frac{N}{3|\v^{\f}_j|}\tilde{\kappa}_{\frac{1}{2},\v^{\f}_j}\Big) \bm{B}_{0,\f} & \mbox{ if }k = 0,\\[0.2em]
				\Big(\sum\limits_{j=1}^3  \frac{N}{3|\v^{\f}_j|}\tilde{\kappa}_{N-\frac{1}{2},\v^{\f}_j}\Big)\bm{B}_{N,\f} & \mbox{ if }k = N,\\[0.2em]
				\Big(\sum\limits_{j=1}^3 \frac{N}{3|\v^{\f}_j|} \big(\tilde{\kappa}_{k-\frac{1}{2},\v^{\f}_j}+\tilde{\kappa}_{k+\frac{1}{2},\v^{\f}_j}\big)\Big) \bm{B}_{k,\f} & \mbox{ otherwise.}
			\end{cases}
		\end{equation*}
		On the other hand, it follows from \eqref{eq:alpha-eta} and \eqref{relationutz} that
		\begin{equation}\nonumber
			\big((\L_{s}\L^*_{t}\bm{\alpha}_1)\otimes \bm{1}_3\big)\odot \bm{B} = \Big(\big(\L_{s}\L^*_{t}2\W_{\mathrm{st},V}^{-1} (\tilde\kappa_{\frac{1}{2},\v_1};\ldots; \tilde\kappa_{\frac{1}{2},\v_{|V|}};\ldots; \tilde\kappa_{N-\frac{1}{2},\v_{|V|}})\big)\otimes \bm{1}_3\Big)\odot\bm{B}.
		\end{equation}
		Then by using the definitions of $\L_{s}$ and $\L_t$ in \eqref{eq:li}, for $(k,\f)\in \ftc$, we get
		\begin{equation*}
			\Big(\big((\L_{s}\L^*_{t}\bm{\alpha}_1)\otimes \bm{1}_3\big)\odot \bm{B}\Big)_{k,\f} =
			\begin{cases}
				\Big(\sum\limits_{\v\in V_{\f}} \frac{N}{3|\v|}\tilde{\kappa}_{\frac{1}{2},\v}\Big) \bm{B}_{0,\f} & \mbox{ if }k = 0,
				\\[0.2em]
				\Big(\sum\limits_{\v\in V_{\f}}  \frac{N}{3|\v|}\tilde{\kappa}_{N-\frac{1}{2},\v}\Big) \bm{B}_{N,\f} & \mbox{ if }k = N,
				\\[0.2em]
				\Big(\sum\limits_{\v\in V_{\f}} \frac{N}{3|\v|} \big(\tilde{\kappa}_{k-\frac{1}{2},\v}+\tilde{\kappa}_{k+\frac{1}{2},\v}\big )\Big) \bm{B}_{k,\f} & \mbox{ otherwise}.
			\end{cases}
		\end{equation*}
		Since $V_{\f} = \{\v_1^{\f},\v_2^{\f},\v_3^{\f}\}$, it follows that $\bm{\alpha}_2 = \big((\L_{s}\L^*_{t}\bm{\alpha}_1)\otimes \bm{1}_3\big)\odot \bm{B}$.
		Consequently, the last line of \eqref{eq:ddot-kkt} holds, which completes the proof.
	\end{proof}

	\section{An inexact proximal ALM for SOCP reformulation}
	\label{sec:alg}
	This section introduces an inexact decomposition-based proximal ALM for solving the SOCP reformulation \eqref{eq:opt-ddot-soc} with highly efficient implementations for the corresponding subproblems.
	The augmented Lagrangian function of \eqref{eq:opt-ddot-soc} is given by
	\begin{equation}\nonumber
		\begin{array}{ll}
			\mathfrak{L}_{\sigma}(\bm{\varphi},\z,\q;\bm{\alpha},\bm{\beta}):= \langle \bm{c}, \bm{\varphi}\rangle_{\mathrm{c},V} + \delta_{\mathbb{Q}}(\T\z)+\langle\mathcal{A}\bm{\varphi}-\q,\bm{\alpha}\rangle_{q}+\langle \z-\B\mathcal{F}\q-\d,\bm{\beta} \rangle_{z} 
			\\[1mm]
			\hspace{8.5em} 
			+ \frac{\sigma}{2}\|\mathcal{A}\bm{\varphi}-\q\|_q^2 + \frac{\sigma}{2}\|\z-\B\mathcal{F}\q-\d\|_z^2,
		\end{array}
	\end{equation}
	where $\|\cdot\|_q$ and $\|\cdot\|_z$ denote the norms induced by the inner products $\langle\cdot,\cdot\rangle_q$ and $\langle\cdot,\cdot\rangle_z$ defined by \eqref{eq:innerqz}.
	Note that the discrete gradient operator $\A$ has a nontrivial kernel.
	Thus, the linear operator $\A^*\A$ is singular. 
	Consequently, minimizing the augmented Lagrangian function with respect to $\bm{\varphi}$ may yield an unbounded solution set, thereby complicating the implementation.
	To address this issue, we further restrict the domain of $\bm{\varphi}$ in \eqref{eq:opt-ddot-soc}.
	Specifically, 
		from \eqref{feaseq} in the proof of \Cref{prop:dotkkt}, one can see that $
		\W_{\mathrm{c},V}\bm c\in\rge(\A^*)$, ensuring that 
		$
		\langle
		\bm{\varphi},
		\bm c
		\rangle_{\mathrm{c},V}
		=
		\langle 
		\Pi_{\rge(\A^*)}(\bm{\varphi}),
		\bm c
		\rangle_{\mathrm{c},V}
		$
		for any $\bm{\varphi}\in\mathbb{R}^{\fvc}$
		due to \eqref{eq:inner-tv}. 
		This, together with the fact that
		$\A\bm{\varphi}=
		\A(\mathrm{\Pi}_{\rge(\A^*)}\bm{\varphi})
		$, 
		implies that $\mathrm{\Pi}_{\rge(\A^*)}\bm{\varphi}$ preserves both the equality constraint involving $\A\bm{\varphi}$ and the objective value.
	Therefore, given a solution $(\bm{\varphi};\q; \z; \bm{\alpha}; \bm{\beta})$ to the KKT system \eqref{eq:ddot-kkt-soc}, it follows that 
	$(\mathrm{\Pi}_{\rge(\A^*)}(\bm{\varphi});\q;\z;\bm{\alpha};\bm{\beta})$ is a solution to both the KKT system \eqref{eq:ddot-kkt-soc} and the KKT system of the following problem
	\begin{equation}
		\label{eq:opt-ddot-soc-re}
		\min_{\bm{\varphi},\q,\z}
		\big\{
		\langle \bm\varphi,\bm{c}\rangle_{\mathrm{c},V} \big\vert  \mathcal{A}\bm{\varphi} = \q,\, \B\F\q+\d =\z,\, \T\z\in\mathbb{Q},\, \bm\varphi\in\rge(\A^*)
		\big\}.
	\end{equation}
	Thus, the KKT system for \eqref{eq:opt-ddot-soc-re} admits a nonempty solution set that is contained in the solution set of \eqref{eq:ddot-kkt-soc}.
	Consequently, we use \eqref{eq:opt-ddot-soc-re} to implement the inexact proximal ALM algorithm from \cite{cl21}, presented in \Cref{alg:socinpalm},
	and the corresponding convergence theorem follows.

	\begin{algorithm}
		\caption{An inexact proximal ALM for solving \eqref{eq:opt-ddot-soc-re}}
		\label{alg:socinpalm}
		\small
		\begin{algorithmic}[1]
			\REQUIRE {A penalty parameter $\sigma > 0$, a dual step length $\tau \in (0, 2)$, and the initial points
				$\bm{q}^{(0)}\in \mathbb{R}^{\fvs}\times \mathbb{R}^{3\times{\ftc}}$,
					$\bm{\alpha}^{(0)}\in \mathbb{R}^{\fvs}\times \mathbb{R}^{3\times{\ftc}}$, and
					$\bm{\beta}^{(0)} \in\mathbb{R}^{\fvs}\times(\mathbb{R}^{3\times{\fts}})^6\times \mathbb{R}^{\fvs}$.
			}
			\ENSURE {The infinite sequence $\{(\bm{\varphi}^{(l)}; \bm{z}^{(l)}; \bm{q}^{(l)};\bm{\alpha}^{(l)}; \bm{\beta}^{(l)})\}$}.
			\FOR{$l=0,1,\ldots$}
			\STATE $(\bm{\varphi}^{(l+1)}, \bm{z}^{(l+1)}) \gets \displaystyle\argmin_{\bm{\varphi}, \bm{z}} \big\{ \mathfrak{L}_{\sigma}(\bm{\varphi}, \bm{z}, \bm{q}^{(l)}; \bm{\alpha}^{(l)}, \bm{\beta}^{(l)}) \mid  \bm{\varphi}\in \rge(\A^*) \big\}$.
			\\[0.2em]
			\STATE $\bm{q}^{(l+1)} \gets \displaystyle\argmin_{\bm{q}} \big\{ \mathfrak{L}_{\sigma}(\bm{\varphi}^{(l+1)}, \bm{z}^{(l+1)}, \bm{q}; \bm{\alpha}^{(l)}, \bm{\beta}^{(l)}) \big\}$.
			\\[0.2em]
			\STATE $\bm{\alpha}^{(l+1)} \gets \bm{\alpha}^{(l)} + \tau\sigma \big(\A\,\bm{\varphi}^{(l+1)} - \bm{q}^{(l+1)}\big)$.
			\\[0.2em]
			\STATE $\bm{\beta}^{(l+1)} \gets \bm{\beta}^{(l)} + \tau\sigma \big(\bm{z}^{(l+1)} - \B\F\,\bm{q}^{(l+1)} - \d\big)$.
			\ENDFOR
		\end{algorithmic}
	\end{algorithm}

	\begin{theorem}
		The whole sequence $\{(\bm{\varphi}^{(l)};\z^{(l)};\q^{(l)};\bm{\alpha}^{(l)};\bm{\beta}^{(l)})\}$ generated by \Cref{alg:socinpalm} converges to a solution of the KKT system \eqref{eq:ddot-kkt-soc}.
	\end{theorem}
	\begin{proof}
		Since $\T$ is nonsingular by \Cref{propk:inverse},
		$\delta_{\mathbb Q}(\T\z)$ is a closed proper convex function of $\z$ by \cite[Theorems 5.7 \& 9.5]{rock-1}. 
		Then, we identify Algorithm \ref{alg:socinpalm} with \cite[Algorithm sGS-inPADMM]{cl21} with $s=1$, $(\bm \varphi,\bm z)$ being the first block-variable and $\q$ being the second block-variable. 
		Recall that the KKT solution set of \eqref{eq:opt-ddot-soc-re} is nonempty and is contained in the solution set of \eqref{eq:ddot-kkt-soc}.
		Moreover, the linear operator $\q\to (-\q;-\B\F\q)$ is injective by definition. 
		Then, both Assumptions 4.1 and 4.2 in \cite{cl21} hold. 
		It follows from \cite[Theorem 4.2(e)]{cl21} that the sequence
		$\{(\bm{\varphi}^{(l)};\z^{(l)};\q^{(l)};\bm{\alpha}^{(l)};\bm{\beta}^{(l)})\}$ converges to a KKT solution of \eqref{eq:opt-ddot-soc-re}, and hence to a solution of \eqref{eq:ddot-kkt-soc}.
	\end{proof}

	\begin{remark}
			\Cref{alg:socinpalm} has the same update form as the classic ADMM \cite{glowinski75,gabay76}.
			However, the dual step length $\tau$ can be chosen from the interval $(0,2)$, whereas the classical ADMM corresponds to $\tau \in (0,\frac{1+\sqrt{5}}{2})$.
			This feature places \Cref{alg:socinpalm} within the framework of inexact proximal ALM based on the symmetric Gauss-Seidel decomposition \cite{lxdsgs}.
		The numerical experiments reported in \cite{cl21} show that larger dual step lengths, such as $\tau=1.9$, can improve numerical efficiency in practice, motivating the use of \Cref{alg:socinpalm} for solving \eqref{eq:opt-ddot-soc-re}. 
	\end{remark}

	\subsection{Implementation details}
	We specify how to solve the subproblems in \Cref{alg:socinpalm} in a computationally economical way.
	Note that the augmented Lagrangian function $\mathfrak{L}_{\sigma}$ is separable in $\bm{\varphi}$ and $\z$, so that the subproblems involving the variables $(\bm{\varphi},\z)$ in \Cref{alg:socinpalm} can be solved independently. In particular, each $\bm{\varphi}^{(l+1)}$ is obtained by solving the linear system
	\begin{equation}
		\label{eq:step-phi}
		(\A_t^*\W_{\mathrm{st},V}\A_t+\A_s^*\W_{\mathrm{c},T}\A_s)\bm{\varphi} = (\A_t^*\W_{\mathrm{st},V}, \A_s^*\W_{\mathrm{c},T}) (\q^{(l)}-\bm{\alpha}^{(l)}/\sigma)-\W_{\mathrm{c},V}\bm{c}/\sigma.
	\end{equation}
	The definition of $\bm{c}$ in \eqref{eq:opt-ddot} ensures that $\W_{\mathrm{c},V}\bm{c}\in\rge(\A^*)$. Consequently, the right-hand side of this linear system also lies in $\rge(\A^*)$. This ensures that the linear system above admits a solution in $\bm\varphi\in\rge(\A^*)$.
	Given the invariance of the linear operator $\A_t^*\W_{\mathrm{st},V}\A_t+\A_s^*\W_{\mathrm{c},T}\A_s$ between iterations, we decouple the full spatial-temporal system \eqref{eq:step-phi} by a spectral decomposition in time, producing a collection of independent surface Poisson systems for which the LU factorizations are precomputed. This efficient technique is consistent with that used in  \cite{hugo18}.
	
	Meanwhile, it is easy to see that
	\begin{equation}
		\label{eq:updatez}
		\z^{(l+1)}=\argmin_{\z}\big \{ \delta_{\mathbb{Q}}(\T\z) + \tfrac{\sigma}{2}\|\z-\x^{(l)}\|^2_{z}\big\},
	\end{equation}
	where $\x^{(l)} := \mathcal{B}\mathcal{F}\q^{(l)}+\d-\bm{\beta}^{(l)}/\sigma$.
	Let $\y: = \T\z$. According to \Cref{propk:inverse},
	one has $\z = \D^{-1}\P^* \y$. Then, \eqref{eq:updatez} can be equivalently transformed  into the following problem
	\begin{equation}
		\label{eq:updateq}
		\min_{\y}\big\{ \delta_{\mathbb{Q}}(\y) + \tfrac{\sigma}{2}\|\y-\T\x^{(l)}\|^2_{\P\W\P^*}
		\big\},
	\end{equation}
	where $\|\cdot\|_{\P\W\P^*}$ is a weighted norm with the weighting matrix $\P\W\P^*$, and
	\begin{equation}
		\label{eq:weight}
		\begin{array}{ll}
			\W
			&=\D^{-1}\diag \big(\W_{\mathrm{st},V},\I_6\otimes \W_{\mathrm{st},T},\W_{\mathrm{st},V}\big)\D^{-1}\\[1mm]
			&=\diag\big(\W_{\mathrm{st},V},
			(\I_2\otimes\diag(\D_1,\D_2,\D_3))^{-1}\\
			&\hspace{9em}
			(\I_6\otimes \W_{\mathrm{st},T})
			(\I_2\otimes\diag(\D_1,\D_2,\D_3))^{-1},
			\W_{\mathrm{st},V}\big)\\[1mm]
			&=\diag\big(\W_{\mathrm{st},V},\I_2\otimes
			\diag(\D_1^{-2}\W_{\mathrm{st},T},
			\D_2^{-2}\W_{\mathrm{st},T},
			\D_3^{-2}\W_{\mathrm{st},T}),
			\W_{\mathrm{st},V}\big).
		\end{array}
	\end{equation}
	Specifically, by the definition of $\mathbb{Q}$ in \eqref{eq:setQ}, the optimal solution to problem \eqref{eq:updateq} can be obtained by solving separate problems
	\begin{equation}
		\label{eq:updatesoc}
		\min_{\y^{k,\v}} \big\{\delta_{\mathbb{K}^{k,\v}_{\text{soc}}}(\y^{k,\v}) + \tfrac{\sigma}{2}\|\y^{k,\v}-\T^{k,\v}\x^{(l)}\|^2_{\P^{k,\v}\W(\P^{k,\v})^*}\big\},
		\quad (k,\v) \in \fvs.
	\end{equation}
	By \eqref{eq:weight}, \eqref{defD}, and \Cref{propk:inverse}, one has $\P^{k,\v}\W(\P^{k,\v})^*=\tfrac{|\v|}{N}\I_{2 + 6|T_{\v}|}$. Consequently, \eqref{eq:updatesoc} is solved by the Euclidean projection $(\y^{(l+1)})^{k,\v}=\Pi_{\mathbb{K}_{\mathrm{soc}}^{k,\v}}\big(\T^{k,\v}\x^{(l)}\big)$, which admits a closed-form solution with a very low complexity (linear in the dimension) \cite[Proposition 3.3]{fukushima},
	and $\z^{(l+1)} := \D^{-1}\P^*\y^{(l+1)}$.
	
	Finally, the update for $\q^{(l+1)}$ is obtained by solving an extremely simple linear system with the coefficient matrix being given by
	\begin{equation*}
		\diag(\W_{\mathrm{st},V},\W_{\mathrm{c},T})+\F^*\B^*\diag(\W_{\mathrm{st},V},\I_6\otimes \W_{\mathrm{st},T},\W_{\mathrm{st},V})\B\F,  
	\end{equation*}
	which is diagonal due to \Cref{prop:diag}.

	\section{Numerical Experiments}
	\label{sec:num}
	In this section, we evaluate the performance of \Cref{alg:socinpalm} for solving the discretized surface DOT problem \eqref{eq:opt-ddot} using the SOCP reformulation \eqref{eq:opt-ddot-soc-re}.
	All numerical experiments were implemented in Python 3.12 on a desktop PC with an Intel Core i9-9900KF CPU (8 cores, 3.60 GHz) and 64 GB of memory.
	The numerical results for the method proposed in this paper can be reproduced using the software package {\tt DOTs-SOCP}, available at GitHub\footnote{\url{https://github.com/chlhnu/DOTs-SOCP}.}.
	
	We measure the accuracy of a numerical solution $(\bm\varphi,\q,\bm\alpha_1,\bm\alpha_2)$ to \eqref{eq:opt-ddot} and its dual via the KKT system \eqref{eq:ddot-kkt}
	by defining the relative error estimator
	\begin{equation}\label{eq:residue-ot}
		\eta_{\mathrm{dot}} \eqdef \max\{\eta_D,\eta_P,\eta_{\text{proj}},\eta_S\},
	\end{equation}
	where
	\begin{equation*}
		\arraycolsep=1.5pt \def\arraystretch{1.5}
		\begin{array}{lll}
			\eta_D \eqdef \frac{\|\A\bm{\varphi}-\q\|_{q}}{C + \|\A\bm{\varphi}\|_{q}+\|\q\|_{q}},  &\quad  
			\eta_P \eqdef \frac{\| \bm{c} + \W_{\mathrm{c},V}^{-1}\A^*(\W_{\mathrm{st},V}\bm{\alpha}_1;\W_{\mathrm{c},T}\bm{\alpha}_2)\|_{\mathrm{c},V}}{C + \|\bm{c}\|_{\mathrm{c},V}},   \\[2mm]
			\eta_{\text{proj}} \eqdef \frac{\|\bm{\alpha}_1-\max\{0,f(\q)+\bm{\alpha}_1\}\|_{\mathrm{st},V}}{C + \|\bm{\alpha}_1\|_{\mathrm{st},V}+\|f(\q)\|_{\mathrm{st},V}}, &\quad 
			  \mbox{and} 
			\quad 
			\eta_{S} \eqdef \frac{\|\bm{\alpha}_2- h(\bm{\alpha}_1,\bm{B})\|_{\mathrm{c},T}}{C + \|\bm{\alpha}_2\|_{\mathrm{c},T}+\|h(\bm{\alpha}_1,\bm{B})\|_{\mathrm{c},T}}.
		\end{array}
	\end{equation*}
	Here $C>0$ is the mean triangle area, $\q \equiv (\bm{A};\bm{B})$, $f(\q)\eqdef \bm{A}+\frac{1}{2}\mathcal{L}_{t}\mathcal{L}_{s}^*(\vertiii{\bm{B}}^2)$, and $h(\bm{\alpha}_1,\bm{B}) \eqdef \big((\mathcal{L}_{s}\mathcal{L}^*_{t}\bm{\alpha}_1) \otimes \bm{1}_3\big) \odot \bm{B}$.
	The tested algorithms terminate when $\eta_{\rm dot}\le\texttt{Tol}$.
	To improve the performance of \Cref{alg:socinpalm}, we adaptively adjust the penalty parameter $\sigma$  based on the primal-dual KKT residuals of \eqref{eq:ddot-kkt-soc} given by
	\begin{equation*}
		\arraycolsep=1.5pt \def\arraystretch{2.2}
		\begin{array}{cc}
			\eta^{\mathrm{soc}}_P\eqdef \max\left\{\eta_D, \frac{\|\z-\B\mathcal{F}\q-\d\|_{z}}{C + \|\d\|_{z}} \right\}
			\quad\mbox{and}\quad
			\eta^{\mathrm{soc}}_D \eqdef\max \left\{\eta_P,\frac{
				\|(\bm\alpha_1;\bm\alpha_2)-\widetilde{\bm\alpha}\|_{q}
			}{
				C + \|(\bm\alpha_1;\bm\alpha_2)\|_{q} + \|\widetilde{\bm\alpha}\|_{q}
			}\right\}, \\
		\end{array}
	\end{equation*}
	where $\widetilde{\bm\alpha} := - \diag(\W_{\mathrm{st},V},\W_{\mathrm{c},T})^{-1} \F^*\B^*\diag (\W_{\mathrm{st},V},\I_6\otimes \W_{\mathrm{st},T},\W_{\mathrm{st},V})\bm{\beta}$.

	\subsection{Test for efficacy via the closed-form solution}
	We first use the following simple example 
	in a flat domain in $\mathbb{R}^2$, for which the ground truth is known, to test the reliability of $\eta_{\mathrm{dot}}$ as an accurate metric for \Cref{alg:socinpalm}. 
	
	\begin{example}
		\label{ex0}
		The Gaussian densities defined on the $x$-$y$ plane
		\begin{equation*}
			\begin{array}{ll}
				\rho_i(x,y) :=
				\mathrm{exp}
				\big(-
				\tfrac{ (x - \nu_i)^2 + (y - \nu_i)^2 }{2 \chi^2}
				\big),
				\quad i=0,1.
			\end{array}
		\end{equation*}
	\end{example}
	The transport between $\rho_0 (x, y)$ and $\rho_1 (x, y)$ in \Cref{ex0} is given by
	$\rho^*(t, x, y) :=
	\mathrm{exp}\big(-\tfrac{(x - c(t))^2 +(y - c(t))^2 }{2 \chi^2}\big)$,
	where $c(t) = (1 - t) \nu_0 + t \nu_1$.
	We set $(\nu_0,\nu_1) := (0.4, 0.6)$ and $\chi := 0.1$,
	and use a triangular mesh of the flat domain $[0,1]^2$ with $9409$ vertices and $18432$ triangles.
	
	We also set four benchmark solvers to compare the numerical performance. 
	The first two are off-the-shelf packages for the original discretized DOT problem \eqref{eq:opt-ddot}, including the ADMM-based software package\footnote{\url{https://github.com/HugoLav/DynamicalOTSurfaces}.}  \cite{hugo18} (denoted by {\tt AD-o}) and the package\footnote{\url{https://github.com/jiajia-yu/FISTA_MFG_mfd}.} based on the proximal gradient method \cite{yjjmanifold} (denoted by {\tt PG}). 
	The remaining two are {\tt Gurobi} (v13.0.2) \cite{gurobi} and {\tt MOSEK} (v11.2.2) \cite{mosek}, applied to the same SOCP reformulation as \Cref{alg:socinpalm}. 
	For both commercial solvers, we keep the default parameters and set the internal termination tolerance as $10^{-6}$.
	While executing \Cref{alg:socinpalm}, we set a few checkpoints based on several KKT residuals and compute $L^{2}$ and $L^{\infty}$ errors between the numerical solution and the ground truth $\rho^*$.
	
	\Cref{tab:versus-exact} presents the $L^{2}$ and $L^{\infty}$ errors for different residuals when executing \Cref{alg:socinpalm}. These algorithms are executed until the time limit ($5000$ seconds) is reached, and the $L^{2}$ and $L^{\infty}$ errors at the final iterations are reported, validating the reliability of $\eta_{\rm dot}$ as a practical measure for accuracy. 
		Meanwhile, both {\tt Gurobi} and {\tt MOSEK} run out of memory on this instance.
		We also noticed that {\tt Gurobi} handles the second-order cone constraints by their equivalent convex quadratic representations, leading to prohibitive memory costs, which is not necessary for further comparison.
		Since the densities of the problems in the surface examples of the next subsection are not strictly positive, we exclude {\tt PG} in subsequent comparisons.
	
	\begin{table}[htbp]
		\footnotesize
		\centering
		\caption{
			Comparison of $L^{2}$-error and $L^{\infty}$-error between exact density $\rho^*(t,x,y)$ and the numerical solutions from \Cref{alg:socinpalm} ({\tt iALM}) under different stopping tolerance (${\tt Tol}= 10^{-3},10^{-4},10^{-5}$), {\tt AD-o}, {\tt PG}, 
			{\tt Gurobi} and {\tt MOSEK}. In the table, ``$\star$'' means termination due to out-of-memory.}
		\begingroup
		\setlength{\tabcolsep}{2.2pt}
		\begin{tabular}{|l|
				c|c|c|c|c|c|c|}
			\hline
			{\bf Algorithm} & {\tt iALM}($10^{-3}$) & {\tt iALM}($10^{-4}$) & {\tt iALM}($10^{-5}$) & {\tt AD-o} & {\tt PG} & {\tt Gurobi} & {\tt MOSEK} \\
			\hline
			$L^2$-error
			& $2.89\mathrm{e}$-2 & $1.17\mathrm{e}$-2 & $5.42\mathrm{e}$-3 & $5.65\mathrm{e}$-3 & $5.56\mathrm{e}$-1
			& \multirow{4}{*}{$\star$} & \multirow{4}{*}{$\star$} \\
			$L^{\infty}$-error
			& $8.29\mathrm{e}$-2 & $3.59\mathrm{e}$-2 & $1.68\mathrm{e}$-3 & $2.04\mathrm{e}$-2 & $8.64\mathrm{e}$-1
			& & \\
			Iterations
			& 297 & 1143 & 5636 & 6177 & 11155
			& & \\
			Time (s)
			& 59.6 & 223.6 & 1161.4 & 5000 & 5000
			& & \\
			\hline
		\end{tabular}
		\endgroup
		\label{tab:versus-exact}
	\end{table}

	\subsection{Numerical results}
	We present details of our numerical experiments and results to illustrate the performance of \Cref{alg:socinpalm} across various DOT problems on diverse surfaces.
	Besides comparing with the ADMM-based software package ({\tt AD-o}) in \cite{hugo18} applied to the discrete DOT problem \eqref{eq:opt-ddot}, we also compare with ADMM applied to the SOCP reformulation \eqref{eq:opt-ddot-soc-re} (denoted by {\tt AD-s}).
	In {\tt AD-s}, the subproblems are solved by the same implementation as in \Cref{alg:socinpalm}, while the parameters follow \cite{hugo18}.  
	For consistency, all these algorithms are terminated when the KKT residual
	$\eta_{\mathrm{dot}}$ in \eqref{eq:residue-ot} falls below a prescribed tolerance.
	The commercial software {\tt MOSEK} for conic optimization is also applied to \eqref{eq:opt-ddot-soc-re}, in which we set the internal solver tolerance to $10^{-6}$. 
	After {\tt MOSEK} terminates, we compute and report the corresponding KKT residual and computational time.
	For all the problems, we set the temporal discretization as $N = 31$, the maximum number of iterations as $5\times10^4$, and the maximum computation time as $10$ hours.
	The examples that we tested are listed below and illustrated in \Cref{fig:illustration} with darker shades representing higher density.
	The triangular meshes for these examples are included in our software package.
	
	\begin{figure}[h]
		\centering
		\subfloat[\Cref{ex1}]{\label{fig:rings}
			\includegraphics[width=0.14\textwidth]{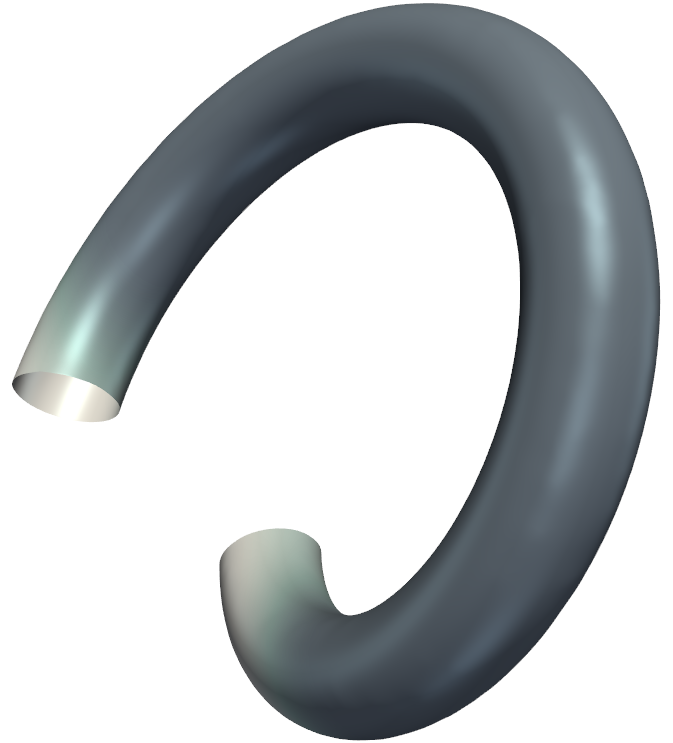}
			\quad 
			\includegraphics[width=0.14\textwidth]{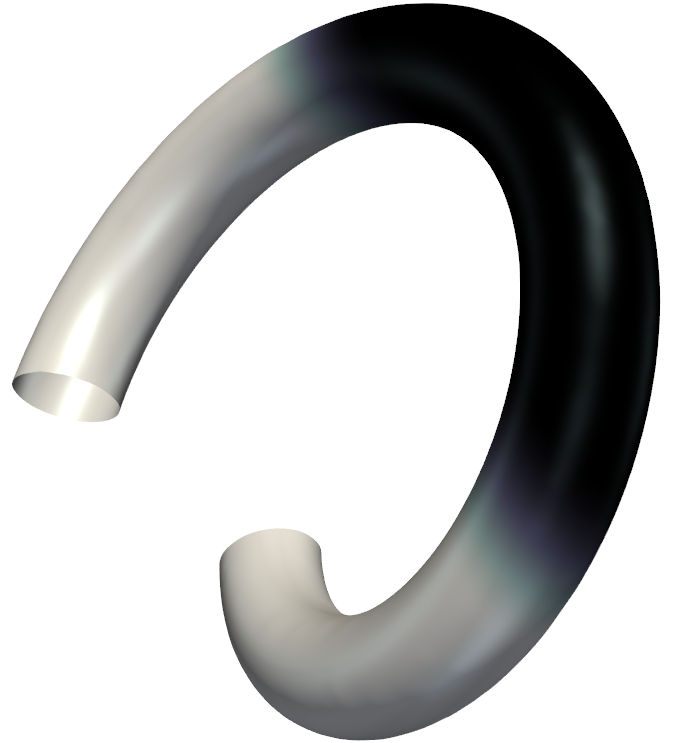}
		}
		\hspace{1cm}
		\subfloat[\Cref{ex2}]{\label{fig:punc}
			\begin{tikzpicture}
				\node (image_node) [anchor = south west, inner sep = 0pt] {\includegraphics[width=0.15\textwidth]{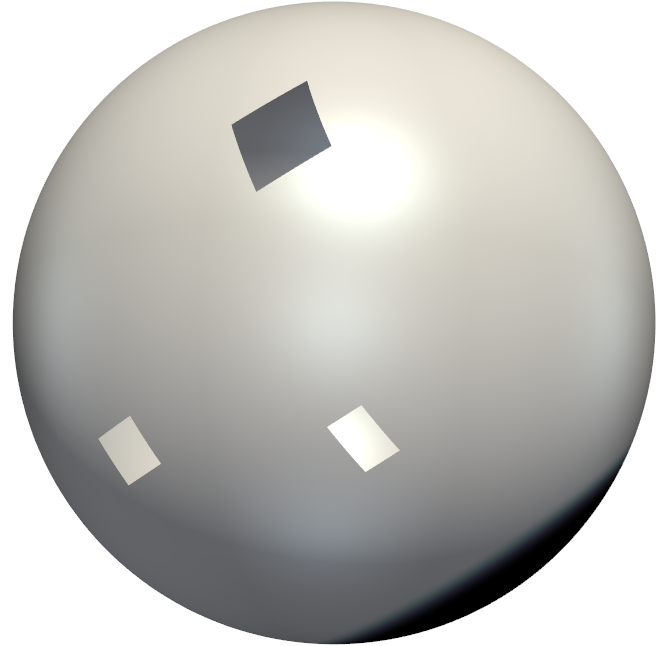}};
				\begin{scope}[
					x = {(image_node.south east)},
					y = {(image_node.north west)},
					font=\footnotesize]
					\node[text = black, align = left, anchor = north west] (removed_text) at (-0.45, 0.8) {\sf Removed};
					\coordinate (arrow_start) at ($(removed_text.east)!0.5!(removed_text.south east)$);
					\coordinate (target_A) at (0.43, 0.76);
					\coordinate (target_B) at (0.20, 0.34);
					\coordinate (target_C) at (0.53, 0.34);
					\tikzset{arrow_style/.style={-latex, black, thin, shorten >=1pt}}
					\draw[arrow_style]  (arrow_start) -- (target_A);
					\draw [arrow_style] (arrow_start) -- (target_B);
					\draw [arrow_style] (arrow_start) -- (target_C);
				\end{scope}
			\end{tikzpicture}
			\includegraphics[width=0.15\textwidth]{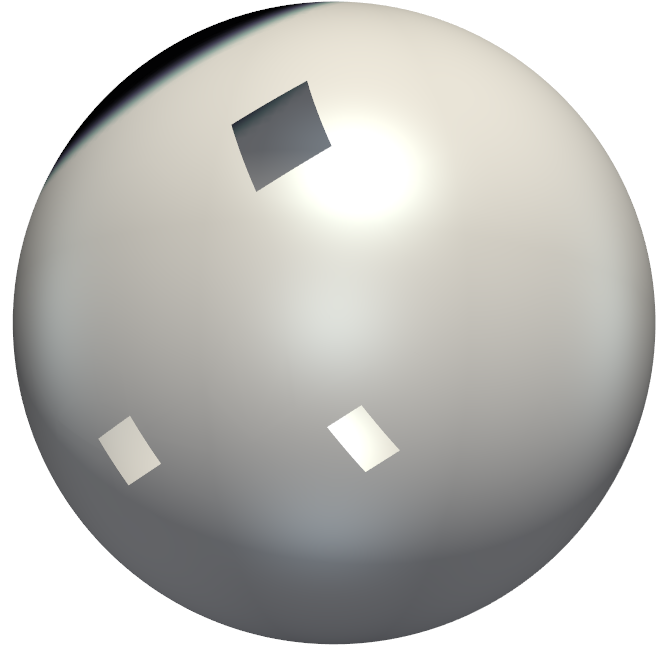}
		}
		\\
		\centering
		\subfloat[\Cref{ex3} (i)]{\label{fig:hand}
			\includegraphics[width=0.16\textwidth]{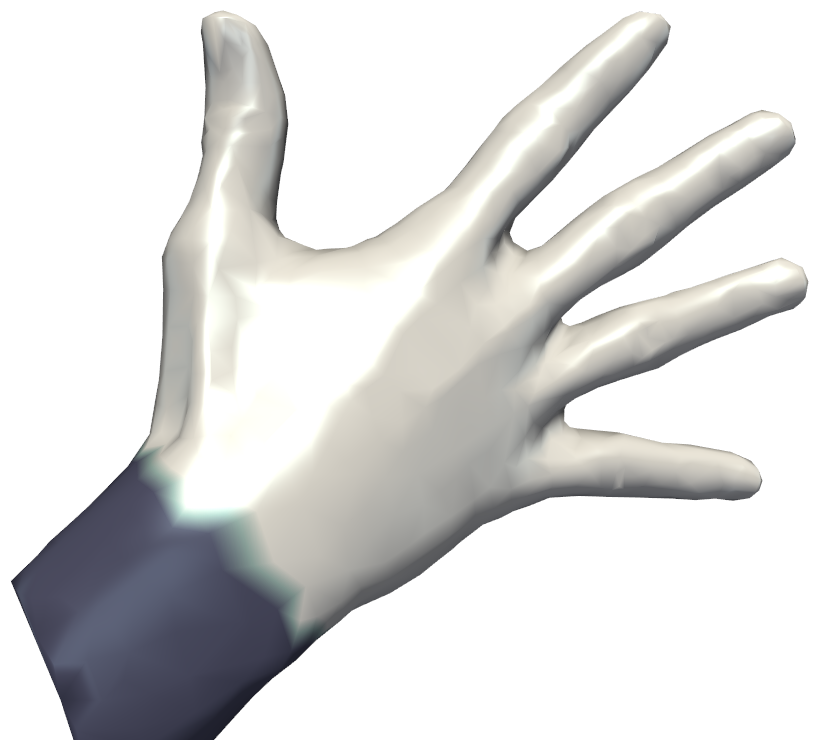}
			\includegraphics[width=0.16\textwidth]{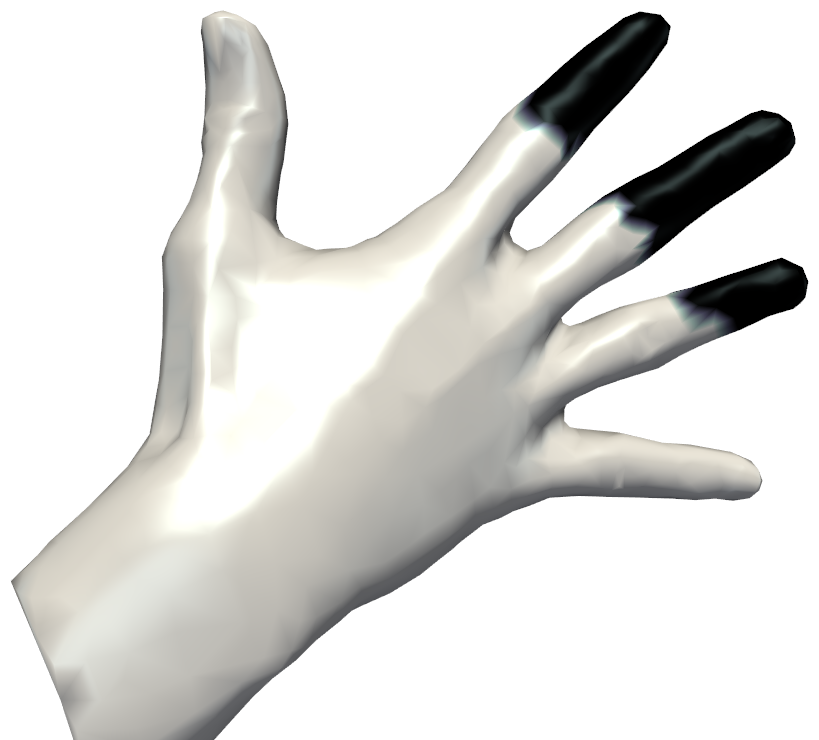}
		}
		\hspace{1cm}
		\subfloat[\Cref{ex3} (ii)]{\label{fig:hand2}
			\includegraphics[width=0.16\textwidth]{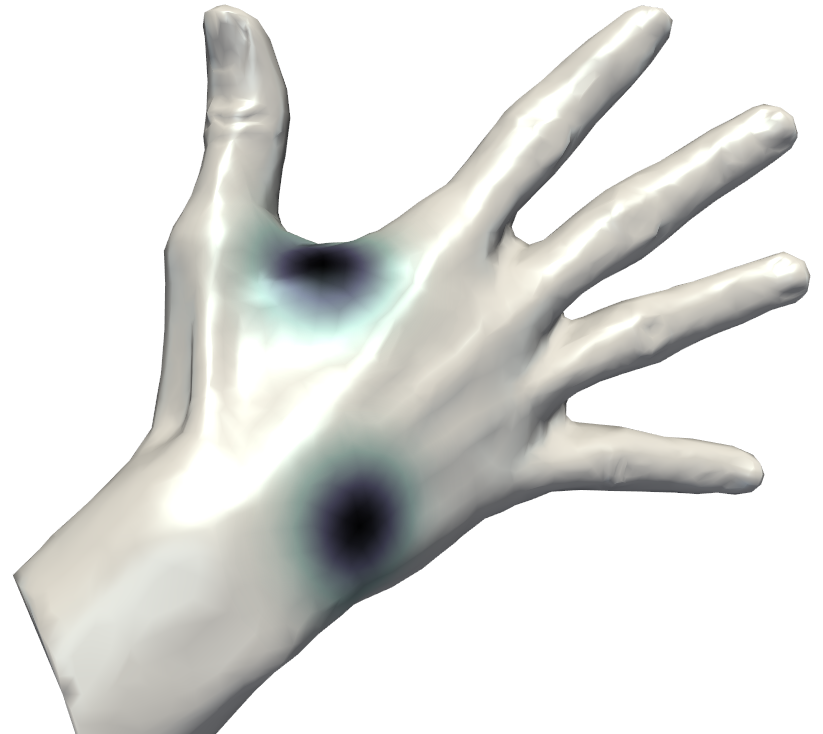}
			\includegraphics[width=0.16\textwidth]{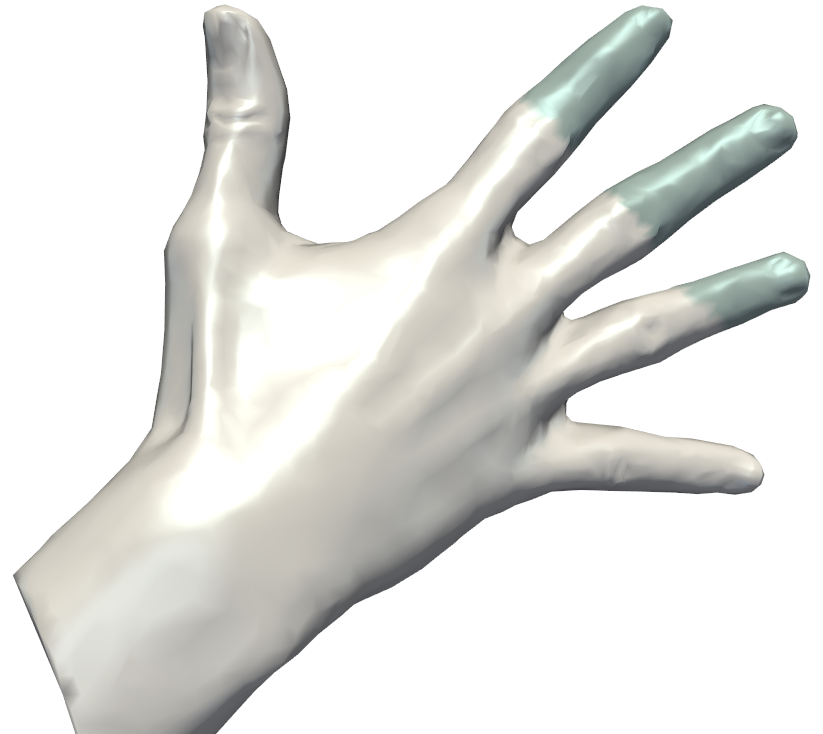}
		}
		\\
		\centering
		\subfloat[\Cref{ex4}]{\label{fig:bunny}
			\includegraphics[width=0.14\textwidth]{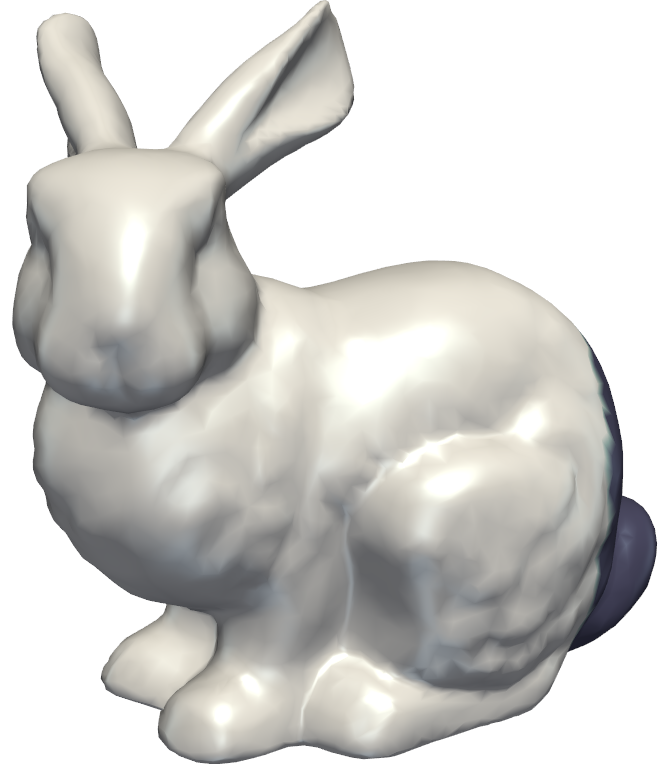}
			\ 
			\includegraphics[width=0.14\textwidth]{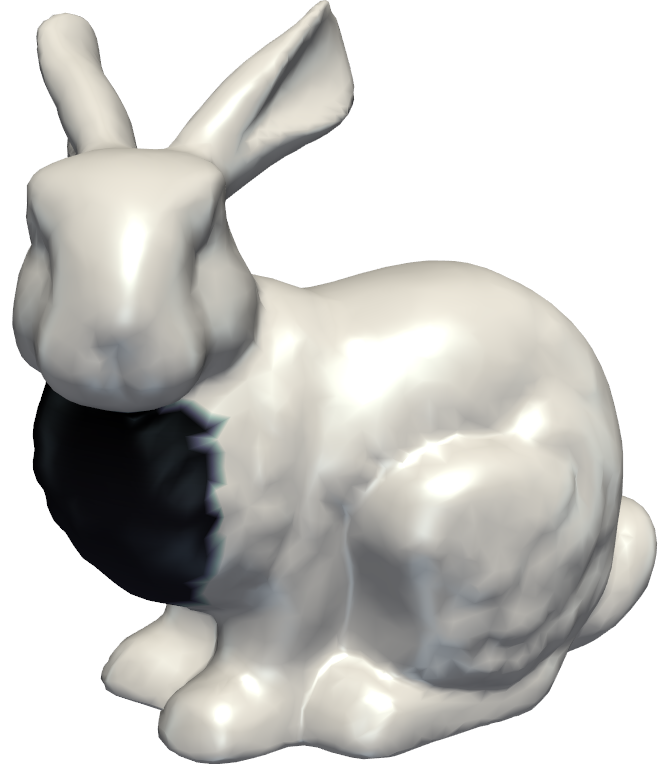}
		}
		\hspace{1cm}
		\subfloat[\Cref{ex5}]{\label{fig:airplane}
			\includegraphics[width=0.17\textwidth]{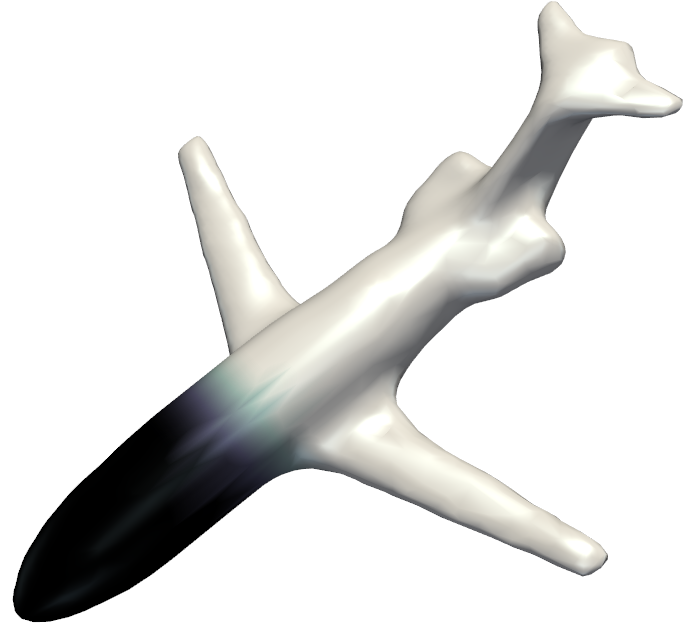}
			\includegraphics[width=0.17\textwidth]{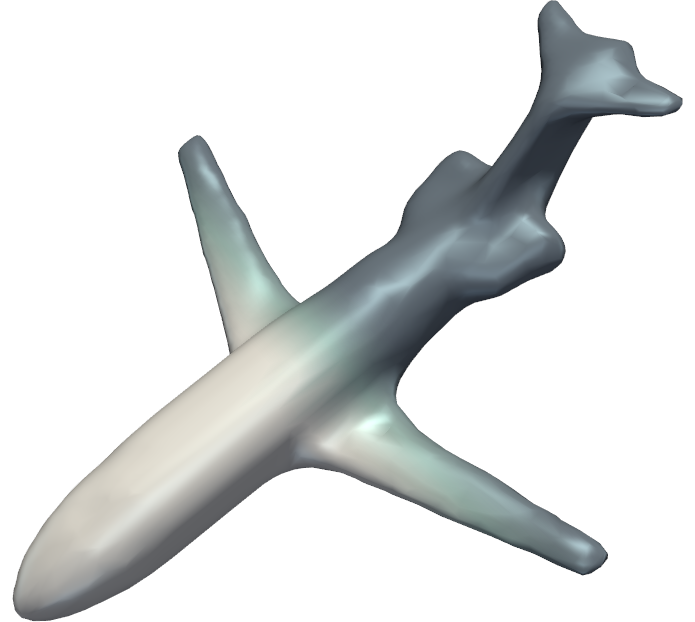}
		}
		\\
		\centering
		\subfloat[\Cref{ex6}]{\label{fig:armadillo}
			\includegraphics[width=0.15\textwidth]{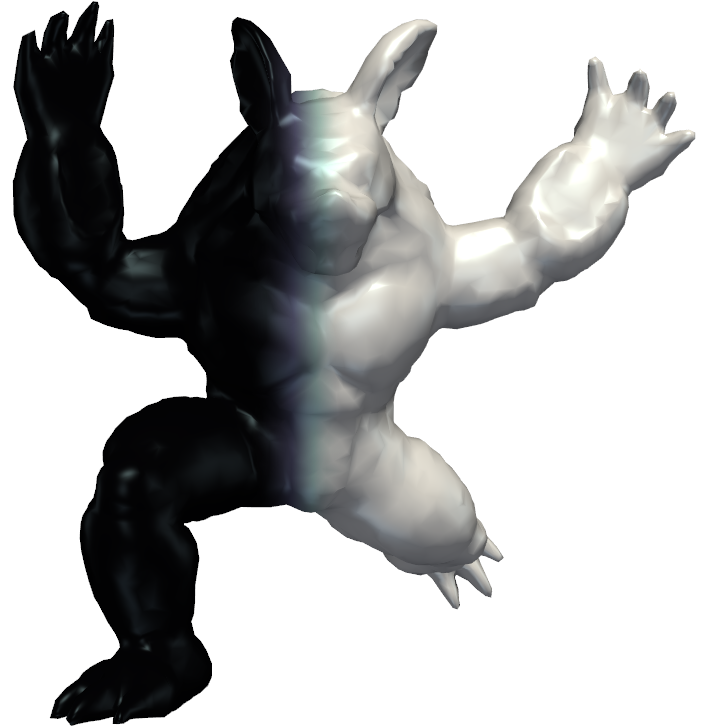}
			\includegraphics[width=0.15\textwidth]{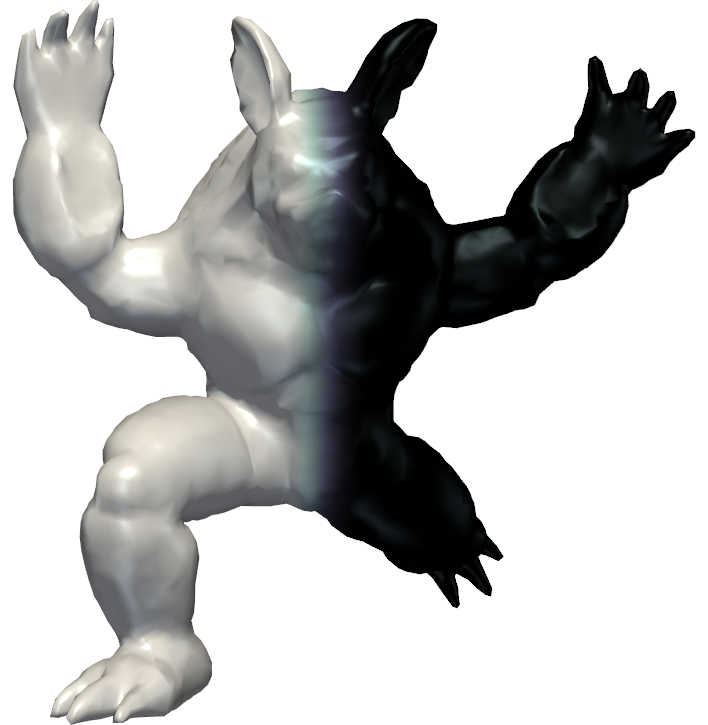}
		}
		\hspace{1cm}
		\subfloat[\Cref{ex-hills}]{\label{fig:hills}
			\includegraphics[width=0.16\textwidth]{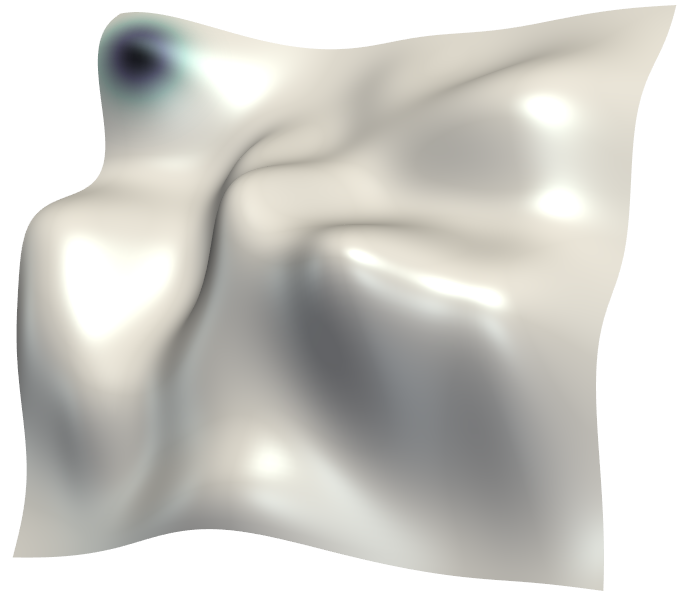}
			\, 
			\includegraphics[width=0.16\textwidth]{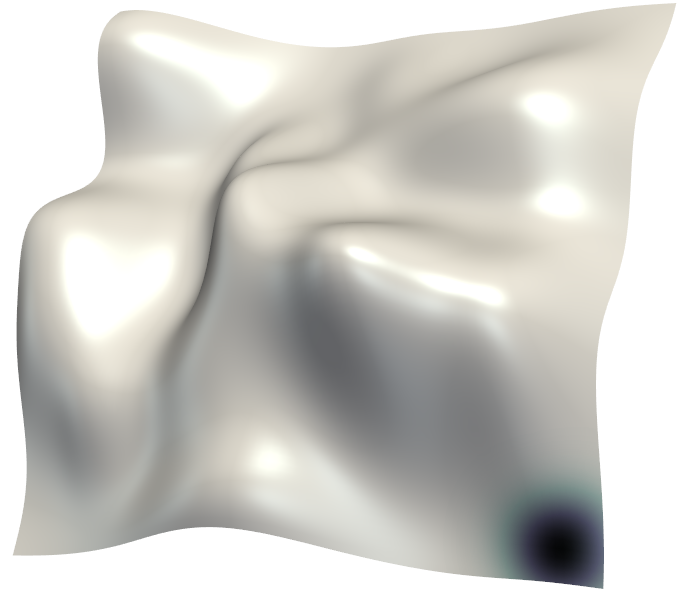}
		}
		\\
		\centering
		\subfloat[\Cref{ex-knots3}]{\label{fig:knot3}
			\includegraphics[width=0.15\textwidth]{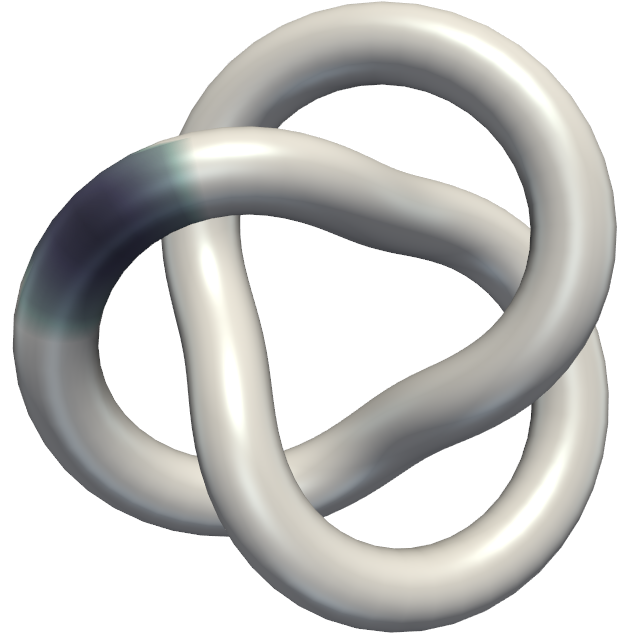}
			\includegraphics[width=0.15\textwidth]{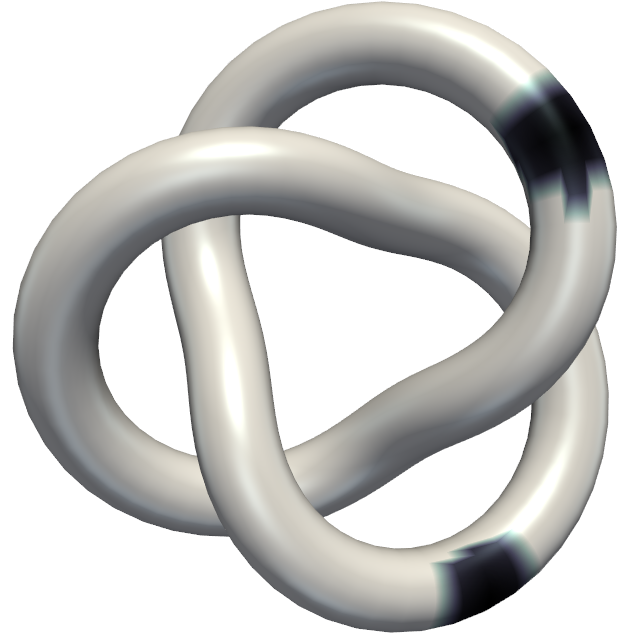}
		}
		\hspace{1cm}
		\subfloat[\Cref{ex-knots5}]{\label{fig:knot5}
			\includegraphics[width=0.16\textwidth]{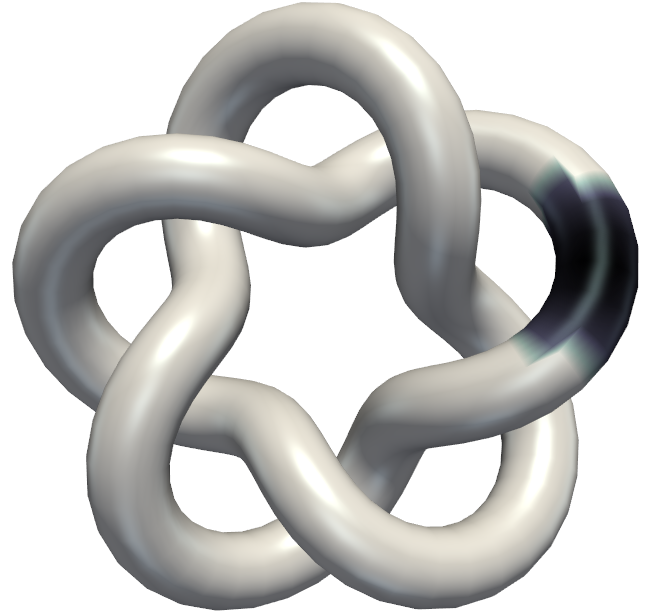}
			\includegraphics[width=0.16\textwidth]{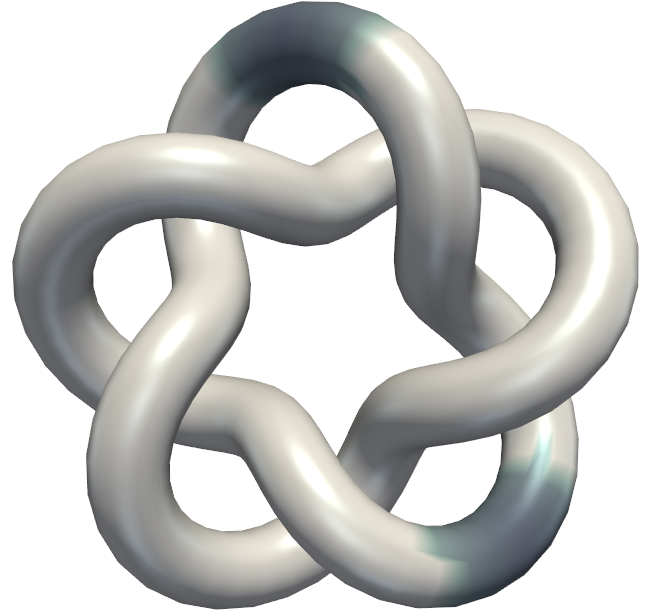}
		}
		\caption{
			Illustration of initial (left) and terminal (right) distributions.}
		\label{fig:illustration}
	\end{figure}
	
	\vspace{-.2\baselineskip}
	\begin{example}
		\label{ex1}
			On a spiral surface, the initial and terminal distributions are shown in \Cref{fig:rings}, discretized with $|V|=9686$ and $|T| = 19288$.
	\end{example}
	
	\vspace{-.7\baselineskip}
	\begin{example}
		\label{ex2}
			On a spherical surface, the initial and terminal distributions are shown in \Cref{fig:punc}.
			In particular, three rectangular regions are removed from the sphere, and the transport path is required to bypass these ``holes''. The surface is discretized with $|V| = 17463$ and $|T|= 34816$.
	\end{example}
	
	\vspace{-.7\baselineskip} 
	\begin{example}
		\label{ex3}
			(i) On the hand surface from \cite{hugo18}, 
			the initial and terminal distributions are shown in \Cref{fig:hand}, with a discretization $|V| = 1515$ and $|T|= 3026$.
			(ii) A similar problem with different distributions, shown in \Cref{fig:hand2}, using a refined discretization with $|V| = 6054$ and $|T| = 12104$.
	\end{example}
	
	\vspace{-.7\baselineskip}
	\begin{example}
		\label{ex4}
			(i) On the bunny surface\footnote{From the Stanford 3D scanning repository (\url{https://graphics.stanford.edu/data/3Dscanrep/}).}, the initial and terminal distributions are shown in \Cref{fig:bunny}. We use a discretization with $|V| = 5047$ and $|T|= 10062$. (ii) A refined discretization with $|V| = 34834$ and $|T| = 69471$.
	\end{example}

	\vspace{-.7\baselineskip}
	\begin{example}\label{ex5}
		On the airplane surface from \cite{hugo18}, the initial and terminal distributions are shown in \Cref{fig:airplane}. We use a discretization with $|V| = 3772$ and $|T| = 7540$, as well as a refined discretization with $|V| = 15082$ and $|T| = 30160$. In particular, the meshes are not evenly distributed.
	\end{example} 
	
	\vspace{-.7\baselineskip}
	\begin{example}
		\label{ex6}
		The transportation problem on the armadillo surface from  \cite{hugo18}. The initial and terminal distributions are shown in \Cref{fig:armadillo}. We use a discretization with $|V| = 5002$ and $|T|=10000$, as well as a refined discretization with $|V| = 20002$ and $|T|=40000$.
	\end{example}
	
	\vspace{-0.7\baselineskip}
	\begin{example}
		\label{ex-hills}
		On a hill surface, the initial and terminal distributions are shown in \Cref{fig:hills}, discretized with $|V| = 10000$ and $|T|=19602$.
	\end{example}

	\vspace{-.7\baselineskip}
	\begin{example}
		\label{ex-knots3}
			On the surface of a trefoil knot, the initial and terminal distributions are shown in \Cref{fig:knot3}, discretized with $|V| = 4096$ and $|T|= 8192$.
	\end{example}
	
	\vspace{-.7\baselineskip}
	\begin{example}
		\label{ex-knots5}
			On the surface of a cinquefoil knot, the initial and terminal distributions are shown in \Cref{fig:knot5}, discretized with $|V| = 4096$ and $|T|= 8192$.
	\end{example}

		\Cref{tab:res} reports the numerical results, in which both $\texttt{Tol} = 10^{-4}$ and $\texttt{Tol} = 10^{-5}$ are used for  
		{\tt AD-o}, {\tt AD-s}, and {\tt iALM}, while {\tt MOSEK} sets the inner tolerance to $10^{-6}$. 
		For $\texttt{Tol} = 10^{-4}$, {\tt AD-o} uses on average about $2.6$ times as many iterations and about $7$ times as much total computational time as \Cref{alg:socinpalm}.
		While {\tt AD-s} improves over {\tt AD-o} by benefiting from the efficient implementation of the SOCP subproblems, it still requires more iterations than \Cref{alg:socinpalm}, leading to an average total computational time about $4.5$ times that of \Cref{alg:socinpalm}.

	\begin{table}[htbp]
		\footnotesize
		\caption{
			Comparison between {\tt AD-o}, {\tt AD-s}, {\tt MOSEK}, and \Cref{alg:socinpalm} ({\tt iALM}). 
			Here, ``{\bf--}'' denotes reaching the time limit of $10$ hours, and ``$\star$'' means termination due to out-of-memory.}
		\label{tab:res}
		\begin{center}
			\begingroup
			\setlength{\tabcolsep}{2.2pt}
			\begin{tabular}{
					|l|
					c
					c
					c
					c
					|
					c
					c
					c
					|
					c@{\hspace{0.1em}}
					c@{\hspace{0.1em}}
					c@{\hspace{0.1em}}
					c|}
				\hline
				Example &
				\multicolumn{4}{c|}{$\eta_{\text{dot}}$} &
				\multicolumn{3}{c|}{Iterations} &
				\multicolumn{4}{c|}{Time (s)} \\
				&
				{\tt AD-o} & {\tt AD-s} & {\tt MOSEK} & {\tt iALM} &
				{\tt AD-o} & {\tt AD-s} & {\tt iALM} &
				{\tt AD-o} & {\tt AD-s} & {\tt MOSEK} & {\tt iALM} \\
				\hline
				
				\ref{ex1}
				& $9.98\mathrm{e}$-5 & $1.00\mathrm{e}$-4 & \multirow{2}{*}{$\star$} & $9.20\mathrm{e}$-5
				& 1472 & 2461 & 653
				& 799.7 & 524.8 & \multirow{2}{*}{$\star$} & \textbf{130.9} \\
				& $1.00\mathrm{e}$-5 & $1.00\mathrm{e}$-5 & & $1.00\mathrm{e}$-5
				& 12640 & 17391 & 2730
				& 6811.1 & 4033.6 & & \textbf{555.3} \\
				\hline
				
				\ref{ex2}
				& $9.94\mathrm{e}$-5 & $1.00\mathrm{e}$-4 & \multirow{2}{*}{$\star$} & $9.66\mathrm{e}$-5
				& 3734 & 5291 & 973
				& 4085.0 & 2511.8 & \multirow{2}{*}{$\star$} & \textbf{416.2} \\
				& $1.00\mathrm{e}$-5 & $1.00\mathrm{e}$-5 & & $1.00\mathrm{e}$-5
				& 18862 & 27208 & 5814
				& 20904.1 & 13437.2 & & \textbf{2473.6} \\
				\hline
				
				\ref{ex3}(i)
				& $9.95\mathrm{e}$-5 & $1.00\mathrm{e}$-4 & \multirow{2}{*}{$4.06\mathrm{e}$-5} & $9.68\mathrm{e}$-5
				& 4615 & 4151 & 1159
				& 357.2 & 132.5 & \multirow{2}{*}{415.7} & \textbf{37.2} \\
				& $1.00\mathrm{e}$-5 & $1.00\mathrm{e}$-5 & & $1.00\mathrm{e}$-5
				& 18031 & 16548 & 5443
				& 1393.5 & 547.7 & & \textbf{175.6} \\
				\hline
				
				\ref{ex3}(ii)
				& $9.97\mathrm{e}$-5 & $1.00\mathrm{e}$-4 & \multirow{2}{*}{$\star$} & $9.84\mathrm{e}$-5
				& 4787 & 6698 & 2043
				& 1600.3 & 794.1 & \multirow{2}{*}{$\star$} & \textbf{237.2} \\
				& $1.00\mathrm{e}$-5 & $1.00\mathrm{e}$-5 & & $1.00\mathrm{e}$-5
				& 40616 & 41746 & 11550
				& 13393.4 & 5256.1 & & \textbf{1333.2} \\
				\hline
				
				\ref{ex4}(i)
				& $9.97\mathrm{e}$-5 & $1.00\mathrm{e}$-4 & \multirow{2}{*}{$\star$} & $9.89\mathrm{e}$-5
				& 3908 & 6018 & 1617
				& 1155.4 & 606.5 & \multirow{2}{*}{$\star$} & \textbf{157.1} \\
				& $1.00\mathrm{e}$-5 & $1.00\mathrm{e}$-5 & & $1.00\mathrm{e}$-5
				& 20200 & 26112 & 7642
				& 5963.8 & 2771.1 & & \textbf{740.0} \\
				\hline
				
				\ref{ex4}(ii)
				& $9.94\mathrm{e}$-5 & $1.00\mathrm{e}$-4 & \multirow{2}{*}{$\star$} & $9.88\mathrm{e}$-5
				& 6783 & 14191 & 2618
				& 14963.3 & 14591.7 & \multirow{2}{*}{$\star$} & \textbf{2423.5} \\
				& $3.36\mathrm{e}$-5 & $2.63\mathrm{e}$-5 & & $1.00\mathrm{e}$-5
				& 15606 & 31573 & 11697
				& {\bf--} & {\bf--} & & \textbf{10944.1} \\
				\hline
				
				\ref{ex5}(i)
				& $9.97\mathrm{e}$-5 & $1.00\mathrm{e}$-4 & \multirow{2}{*}{$3.32\mathrm{e}$-5} & $9.85\mathrm{e}$-5
				& 5607 & 7972 & 1718
				& 1084.8 & 715.6 & \multirow{2}{*}{2200.5} & \textbf{126.5} \\
				& $1.00\mathrm{e}$-5 & $1.00\mathrm{e}$-5 & & $1.00\mathrm{e}$-5
				& 40469 & 31105 & 10266
				& 7873.1 & 2711.8 & & \textbf{742.6} \\
				\hline
				
				\ref{ex5}(ii)
				& $9.97\mathrm{e}$-5 & $1.00\mathrm{e}$-4 & \multirow{2}{*}{$\star$} & $9.98\mathrm{e}$-5
				& 5950 & 15746 & 2707
				& 5295.0 & 6198.7 & \multirow{2}{*}{$\star$} & \textbf{984.0} \\
				& $1.07\mathrm{e}$-5 & $1.00\mathrm{e}$-5 & & $1.00\mathrm{e}$-5
				& 39747 & 67751 & 15739
				& {\bf--} & 26534.3 & & \textbf{5608.3} \\
				\hline
				
				\ref{ex6}(i)
				& $9.97\mathrm{e}$-5 & $1.00\mathrm{e}$-4 & \multirow{2}{*}{$4.33\mathrm{e}$-5} & $9.98\mathrm{e}$-5
				& 4201 & 8898 & 1687
				& 1214.8 & 924.5 & \multirow{2}{*}{2963.4} & \textbf{164.6} \\
				& $1.00\mathrm{e}$-5 & $1.00\mathrm{e}$-5 & & $1.00\mathrm{e}$-5
				& 31187 & 39813 & 8477
				& 8819.8 & 4689.9 & & \textbf{822.6} \\
				\hline
				
				\ref{ex6}(ii)
				& $1.00\mathrm{e}$-4 & $1.00\mathrm{e}$-4 & \multirow{2}{*}{$\star$} & $9.86\mathrm{e}$-5
				& 7175 & 16147 & 2910
				& 8804.9 & 8975.0 & \multirow{2}{*}{$\star$} & \textbf{1456.8} \\
				& $1.67\mathrm{e}$-5 & $1.38\mathrm{e}$-5 & & $1.00\mathrm{e}$-5
				& 27908 & 58888 & 13847
				& {\bf--} & {\bf--} & & \textbf{6958.4} \\
				\hline
				
				\ref{ex-hills}
				& $9.98\mathrm{e}$-5 & $1.00\mathrm{e}$-4 & \multirow{2}{*}{$\star$} & $9.85\mathrm{e}$-5
				& 4075 & 6296 & 2055
				& 2354.9 & 1419.2 & \multirow{2}{*}{$\star$} & \textbf{419.9} \\
				& $1.00\mathrm{e}$-5 & $1.00\mathrm{e}$-5 & & $1.00\mathrm{e}$-5
				& 29690 & 36181 & 10852
				& 17117.5 & 8436.8 & & \textbf{2259.4} \\
				\hline
				
				\ref{ex-knots3}
				& $9.99\mathrm{e}$-5 & $1.00\mathrm{e}$-4 & \multirow{2}{*}{$1.89\mathrm{e}$-5} & $9.88\mathrm{e}$-5
				& 6819 & 7516 & 3152
				& 1527.9 & 618.3 & \multirow{2}{*}{3173.4} & \textbf{254.6} \\
				& $1.00\mathrm{e}$-5 & $1.00\mathrm{e}$-5 & & $1.00\mathrm{e}$-5
				& 37733 & 34621 & 15136
				& 8387.1 & 2956.5 & & \textbf{1225.2} \\
				\hline
				
				\ref{ex-knots5}
				& $9.97\mathrm{e}$-5 & $1.00\mathrm{e}$-4 & \multirow{2}{*}{$4.09\mathrm{e}$-5} & $9.88\mathrm{e}$-5
				& 7224 & 11893 & 4222
				& 1622.4 & 1000.6 & \multirow{2}{*}{3321.0} & \textbf{340.2} \\
				& $1.00\mathrm{e}$-5 & $1.00\mathrm{e}$-5 & & $1.00\mathrm{e}$-5
				& 35000 & 44819 & 15827
				& 7878.2 & 3904.5 & & \textbf{1282.4} \\
				\hline
			\end{tabular}
			\endgroup
		\end{center}
	\end{table}

	For the more stringent tolerance $\texttt{Tol} = 10^{-5}$, {\tt AD-o} fails to reach the prescribed accuracy within $10$ hours on several larger instances.
		Compared with \Cref{alg:socinpalm}, {\tt AD-o} uses, on average, about $8.9$ times as much total computational time on the instances it solves within $10$ hours.
		For the SOCP reformulation, {\tt AD-s} reaches $\texttt{Tol} = 10^{-5}$ on one more instance than {\tt AD-o}, namely \Cref{ex5} with $|V| = 15082$, whereas \Cref{alg:socinpalm} remains about $4.3$ times faster on average among the instances solved by {\tt AD-s} within the time limit.
		\Cref{tab:res} also shows that {\tt MOSEK} is competitive on several smaller meshes, but it terminates due to insufficient memory on the larger instances. In all instances for which {\tt MOSEK} returns a solution, \Cref{alg:socinpalm} reaches the prescribed tolerance with substantially less computational time.
		
		The computed displacement interpolations by \Cref{alg:socinpalm} with $\texttt{Tol} = 10^{-5}$ for \Cref{ex2,ex-hills} are shown in \Cref{fig:snapshot-punc,fig:snapshot-hills}.
		In \Cref{fig:snapshot-punc}, the mass moves from the lower part to the upper part while bypassing the removed regions.
		Additional snapshots of the interpolations for other examples, obtained via \Cref{alg:socinpalm}, are provided in Section 1 of the supplementary document\footnote{\url{https://github.com/chlhnu/DOTs-SOCP/releases/download/v1.0/Document.pdf}.}.

	\begin{figure}[htbp]
		\centering
		\includegraphics[width=0.16\linewidth]{images/refined_punctured_ball_mu0.png}
		\ 
		\includegraphics[width=0.16\linewidth]{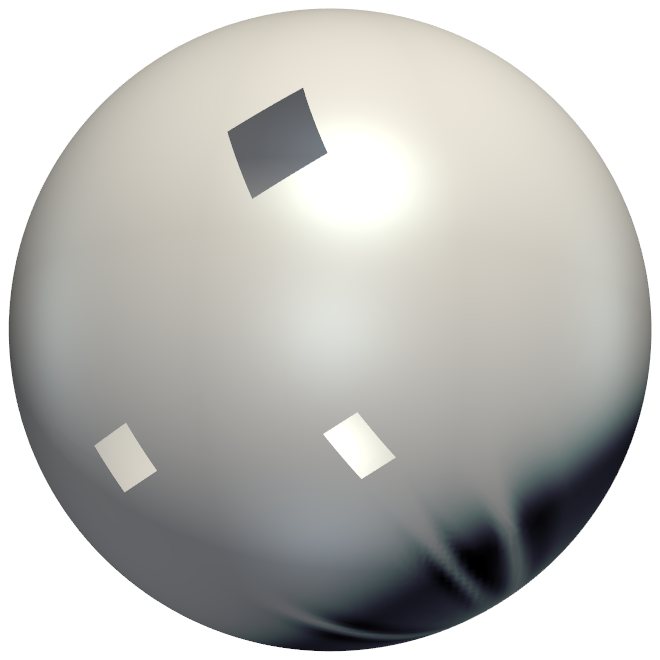}
		\ 
		\includegraphics[width=0.16\linewidth]{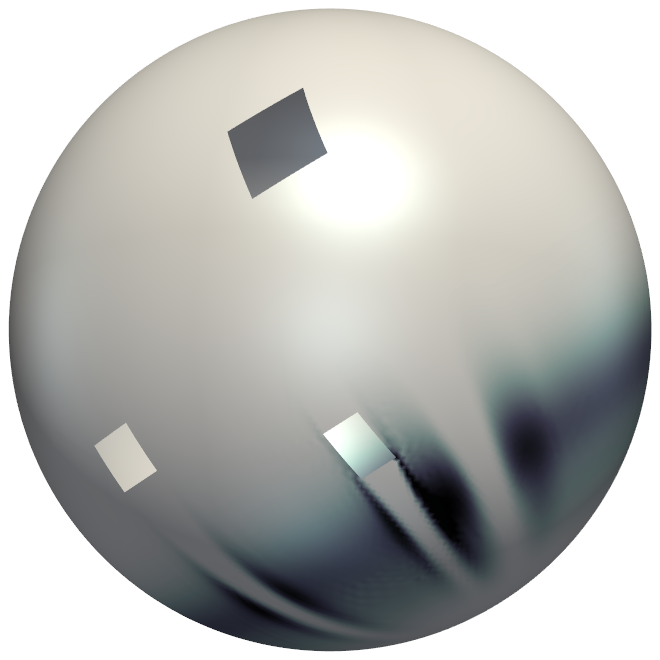}
		\ 
		\includegraphics[width=0.16\linewidth]{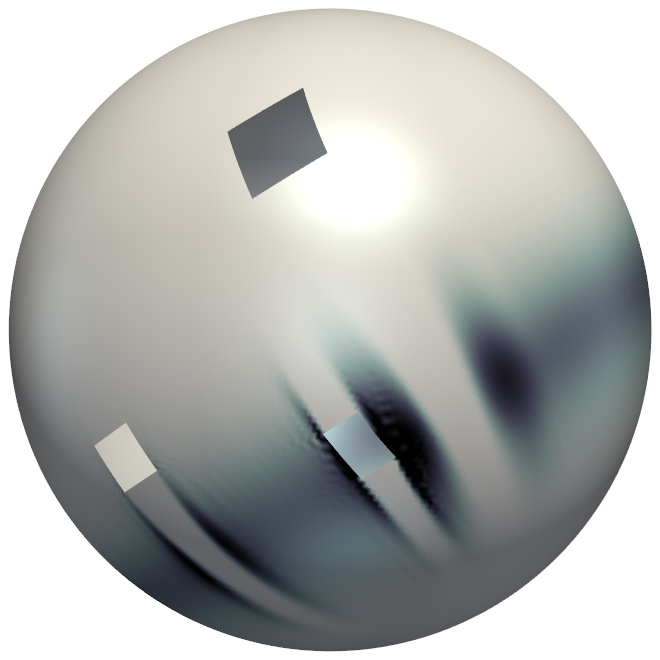}
		\ 
		\includegraphics[width=0.16\linewidth]{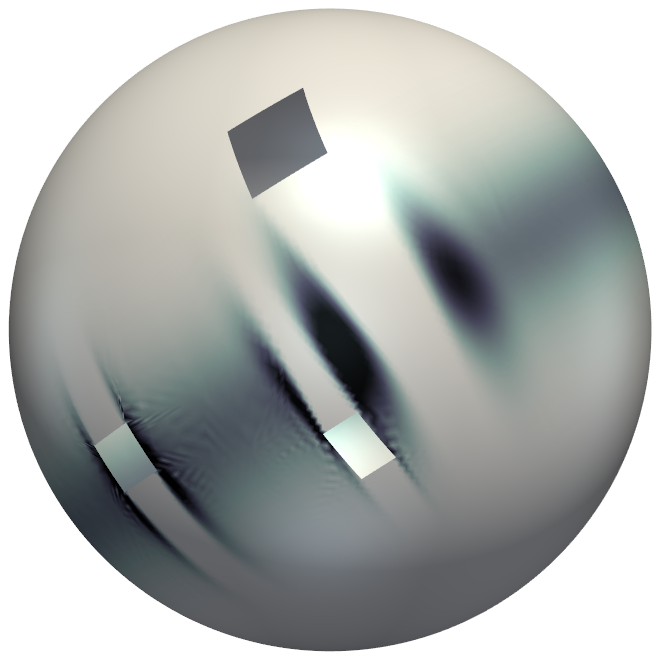}
		\\[1mm]
		\hspace{0.001\linewidth}
		\centering
		\includegraphics[width=0.16\linewidth]{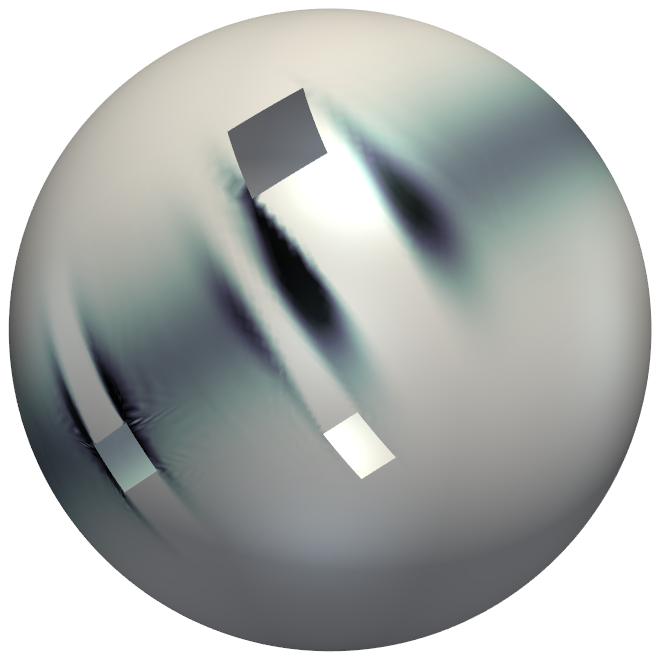}
		\ 
		\includegraphics[width=0.16\linewidth]{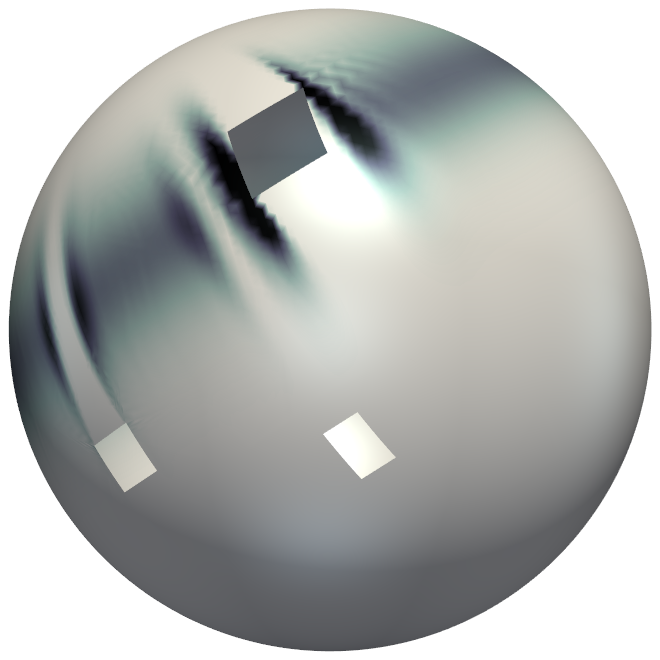}
		\ 
		\includegraphics[width=0.16\linewidth]{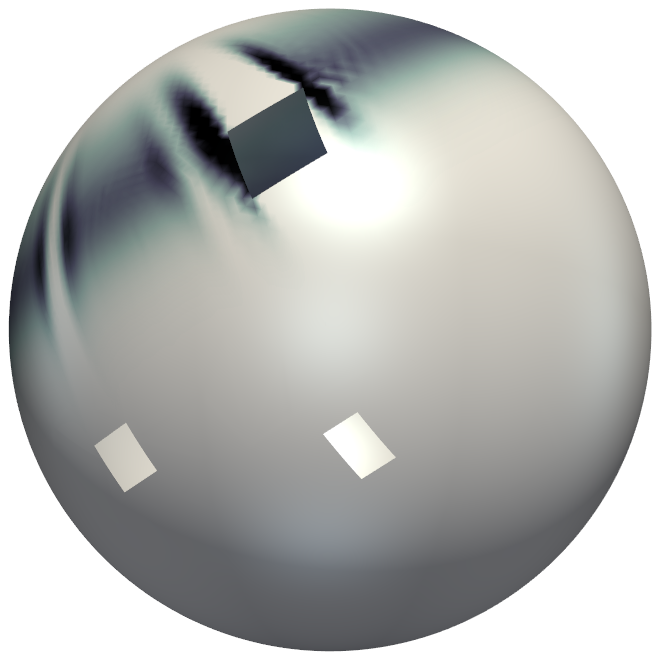}
		\ 
		\includegraphics[width=0.16\linewidth]{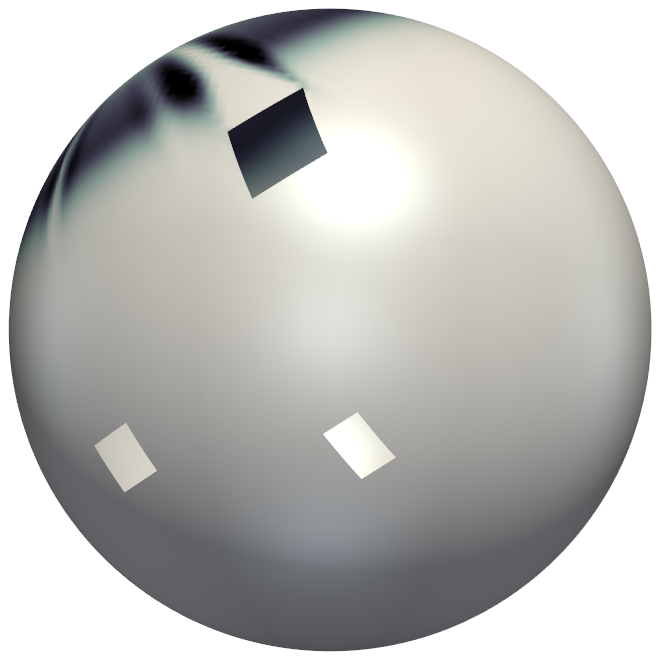}
		\ 
		\includegraphics[width=0.16\linewidth]{images/refined_punctured_ball_mu1.png}
		\caption{Computed displacement interpolation for \Cref{ex2}.}
		\label{fig:snapshot-punc}
	\end{figure}

	\begin{figure}[htbp]
		\centering
		\includegraphics[width=0.17\linewidth]{images/hills_mu0.png}
		\includegraphics[width=0.17\linewidth]{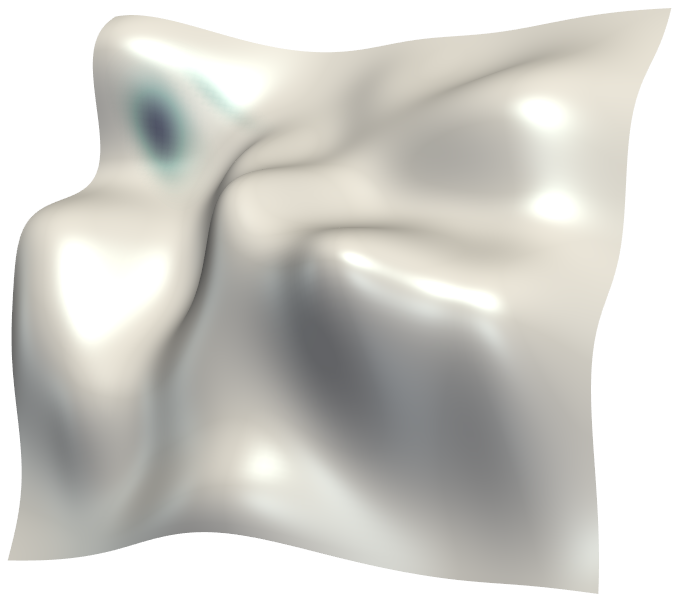}
		\includegraphics[width=0.17\linewidth]{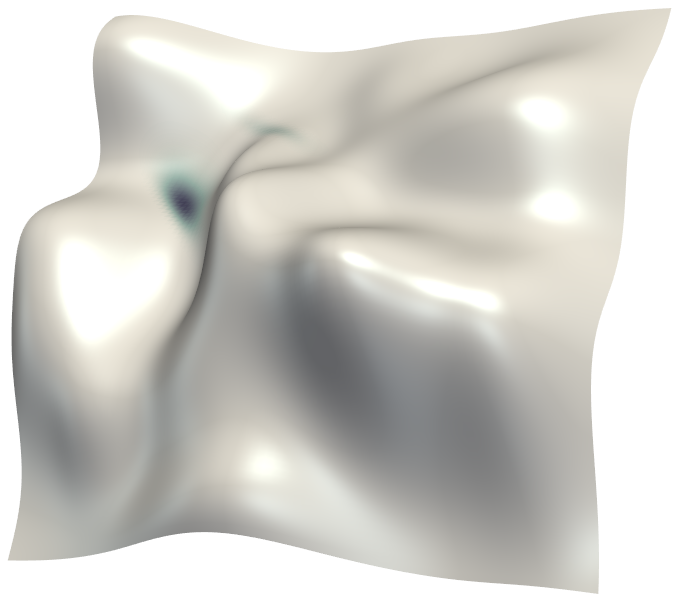}
		\includegraphics[width=0.17\linewidth]{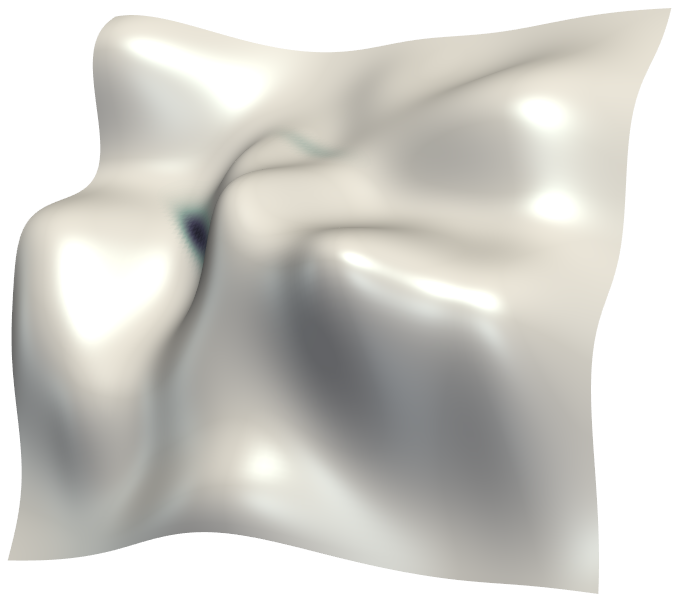}
		\includegraphics[width=0.17\linewidth]{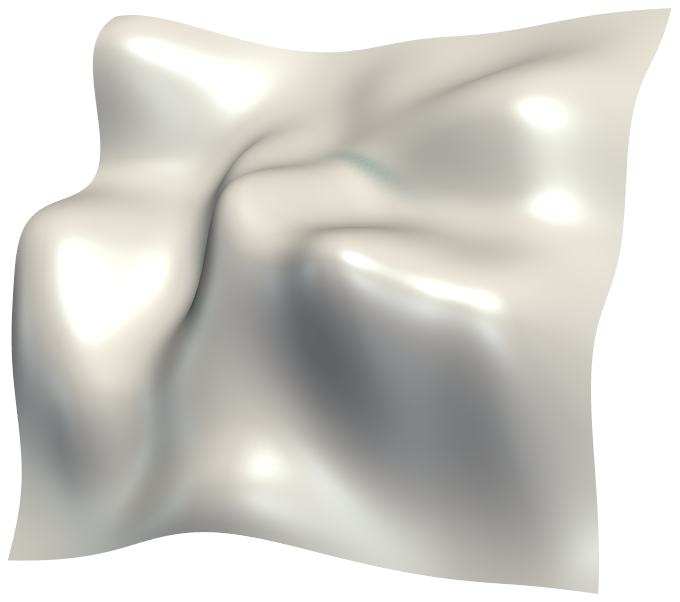}
		\\
		\includegraphics[width=0.17\linewidth]{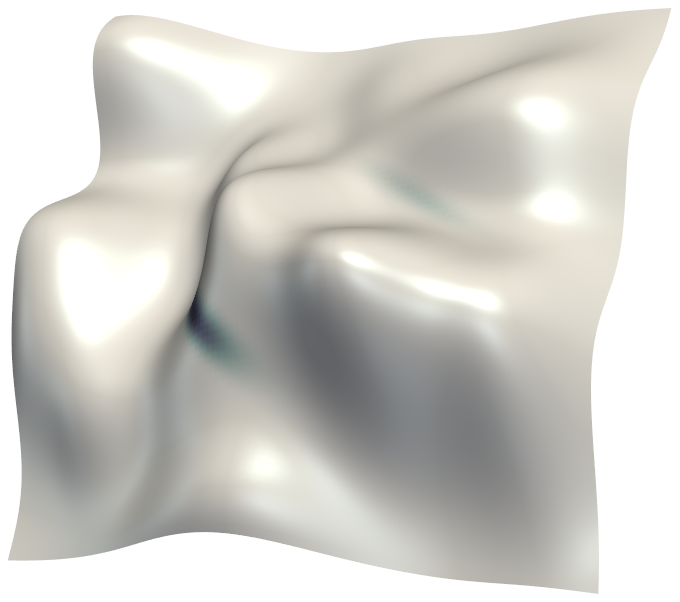}
		\includegraphics[width=0.17\linewidth]{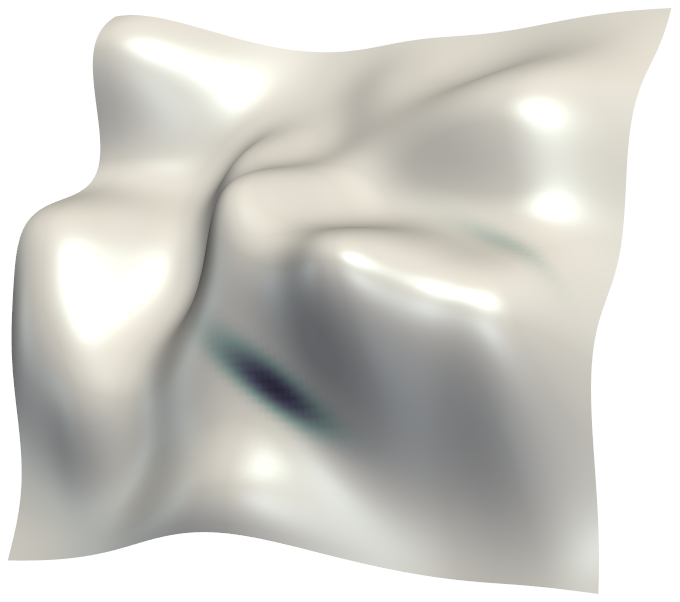}
		\includegraphics[width=0.17\linewidth]{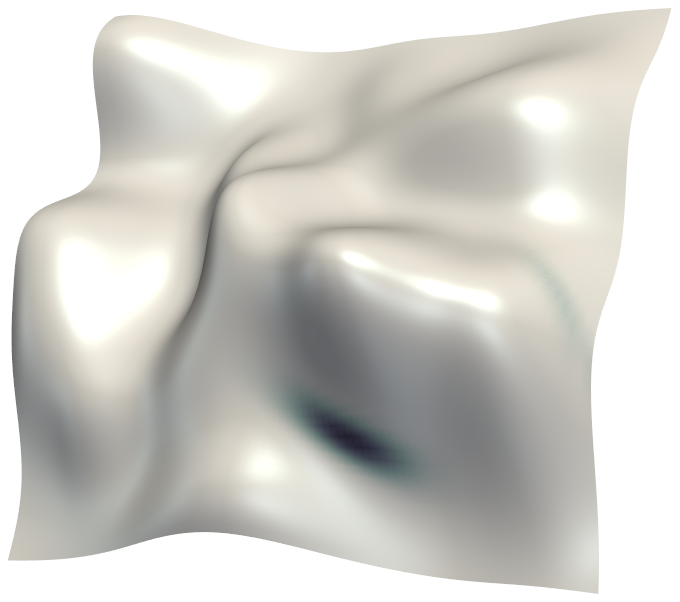}
		\includegraphics[width=0.17\linewidth]{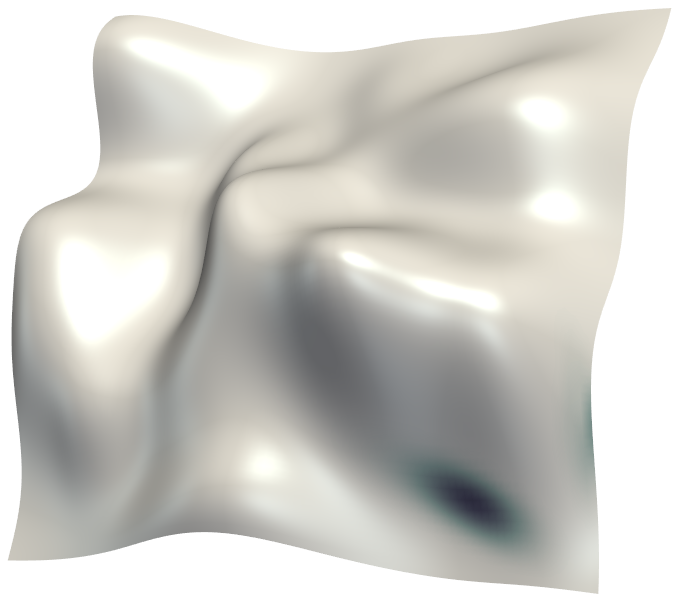}
		\includegraphics[width=0.17\linewidth]{images/hills_mu1.png}
		\caption{Computed displacement interpolation for \Cref{ex-hills}. }
		\label{fig:snapshot-hills}
	\end{figure}

		In conclusion, the numerical results show that the efficiency of \Cref{alg:socinpalm} comes from the economical SOCP implementation, the inexact proximal ALM framework of \cite{cl21} (which utilizes a larger dual step length $\tau$ with proven practical benefits), and an adaptive adjustment of the penalty parameter.

	\subsection{A practical approach for congestion phenomenon}
	The numerical results in \Cref{tab:res} show that the performance of \Cref{alg:socinpalm} deteriorates slightly on \Cref{ex-hills,ex-knots3,ex-knots5}. These instances require a considerable number of iterations to achieve the prescribed tolerance, despite having moderate mesh sizes, with $|V|\leq 10^4$.
		As shown in \Cref{fig:snapshot-hills} for \Cref{ex-hills}, the mass is concentrated along a few narrow ``valley'' regions during the evolution.
		A similar concentration phenomenon is also observed for \Cref{ex-knots3,ex-knots5}, whose snapshots of interpolations are provided in Sections 1.9 and 1.10 of the supplementary document.
		These localized high-density regions create a congestion effect, increasing the difficulty of solving the corresponding optimization problems.
	For such congestion-prone examples, one can follow the discussion in \cite[Section 5.4]{hugo18} to introduce an extra perturbation variable $\bm{\varrho}\in \mathbb{R}^{\fvs}$ and penalize the perturbation, obtaining the modification of the discretized problem \eqref{eq:opt-ddot} as
	\begin{equation}
		\label{eq:opt-ddot-cong}
		\min_{\bm{\varphi},\q,\bm{\varrho}}
		\left\{\langle \bm{\varphi},\bm{c}\rangle_{\mathrm{c},V} + \tfrac{1}{2\theta} \|\bm{\varrho}\|_{\mathrm{st},V}^2
		\ \Big\vert\
		\begin{array}{ll}
			\mathcal{A}\bm{\varphi} -  (\bm{\varrho}; \bm{0}_{3|\ftc|})  = \q\equiv (\bm{A};\bm{B}),
			\\[1pt]
			\bm{A} + \frac{1}{2} \L_{t} \L^*_{s} (\vertiii{\bm{B}}^2)\le 0
		\end{array}
		\right\},
	\end{equation}
	where $\theta>0$ controls the strength of the congestion phenomenon. A larger value of $\theta$ imposes a stronger congestion penalty in \eqref{eq:opt-ddot-cong} and typically leads to a less concentrated transport path.
	We can still utilize the SOCP reformulation proposed in this paper, 
	along with a marginally modified \Cref{alg:socinpalm}, to solve \eqref{eq:opt-ddot-cong}.
	
	\Cref{tab:cong} reports the numerical results of the inexact proximal ALM
	(implemented similarly to \Cref{alg:socinpalm}) in solving \eqref{eq:opt-ddot-cong} for \Cref{ex-hills,ex-knots3,ex-knots5} under different values of $\theta$  with $\texttt{Tol}=10^{-4}$. 
		According to \Cref{tab:cong}, the proposed framework successfully solved all tested values of $\theta$. 
		In particular, $\theta = 0.01$ is sufficient to substantially reduce the computational difficulty while producing a density evolution close to that of the standard DOT model, as shown in Section 2 of the supplementary document.
		This suggests that adding a small positive $\theta$ constitutes a practical approach that preserves the accuracy of the interpolation while improving computational efficiency.
	\begin{table}[htbp]
		\footnotesize
		\centering
		\caption{Performance of adapting \Cref{alg:socinpalm} for solving \eqref{eq:opt-ddot-cong} under different congestion parameters.}
		\begingroup
		\setlength{\tabcolsep}{2.2pt}
		\begin{tabular}{|l|c|c|c|}
			\hline
			{\bf Example} & $\theta$ & Iterations & Time (s) \\
			\hline
			\ref{ex-hills}
			& $0 \mid 0.01 \mid 0.05$
			& $2055 \mid 364 \mid 298$
			& $419.9 \mid 77.2 \mid 64.4$ \\
			\hline
			\ref{ex-knots3}
			& $0 \mid 0.01 \mid 0.05$
			& $3152 \mid 661 \mid 277$
			& $254.6 \mid 54.8 \mid 22.0$ \\
			\hline
			\ref{ex-knots5}
			& $0 \mid 0.01 \mid 0.05$
			& $4222 \mid 1571 \mid 739$
			& $340.2 \mid 127.1 \mid 60.9$ \\
			\hline
		\end{tabular}
		\endgroup
		\label{tab:cong}
	\end{table}

	\section{Conclusions}
	\label{sec:con}
	In this paper, we proposed an efficient numerical optimization framework for solving DOT problems on two-dimensional Riemannian manifolds.
	Building upon the convex DOT model of Benamou-Brenier-Lisini, we
	introduced a novel reformulation of the discretized dual DOT problem into a linear SOCP problem, which yields a tractable model by effectively decoupling variables within the constraints. 
	To solve the resulting SOCP, we developed an inexact proximal ALM that admits a highly efficient numerical implementation and is theoretically guaranteed to converge to a KKT solution without imposing any additional assumptions.
	Extensive numerical experiments validated the effectiveness, robustness, and superior computational efficiency of our method, demonstrating significant computational advantages over state-of-the-art surface DOT solvers and commercial software {\tt Gurobi} and {\tt MOSEK} for conic optimization. 
	Furthermore, the proposed framework has been released as an open-source software package available on GitHub.
	Beyond the current scope, this approach holds promising potential for addressing a broader class of optimization problems in Wasserstein spaces, such as Wasserstein gradient flows and mean-field games, thereby opening interesting avenues for future investigation.

	\appendix
	\section{Discrete Riemannian gradients of hat functions}
	\label{sec:gradient}
	Given the triangle $\f$ with three vertices $\v_1^{\f}$, $\v_2^{\f}$, and $\v_3^{\f}$, define the affine function
	\begin{equation*}
		G_{\f}(\varepsilon_1,\varepsilon_2):= \v_1^{\f} 
		+\varepsilon_1(\v_2^{\f}-\v_1^{\f})
		+\varepsilon_2(\v_3^{\f}-\v_1^{\f}),
		\quad
		0\leq \varepsilon_1,\varepsilon_2,\quad \varepsilon_1+\varepsilon_2\leq 1.
	\end{equation*}
	Since the hat functions $h_{\v_i^{\f}},$ $i=1,2,3$ are linear on $\f$, it follows that
	\begin{equation*}
		h_{\v_i^{\f}}\circ G_{\f}(\varepsilon_1,\varepsilon_2) =
		h_{\v_i^{\f}}(\v_1^{\f})+
		\varepsilon_1(h_{\v_i^{\f}}(\v_2^{\f})-h_{\v_i^{\f}}(\v_1^{\f}))+\varepsilon_2(h_{\v_i^{\f}}(\v_3^{\f})-h_{\v_i^{\f}}(\v_1^{\f})).
	\end{equation*}
	Specifically, one has $h_{\v_1^{\f}}\circ G_{\f}(\varepsilon_1,\varepsilon_2) = 1-\varepsilon_1-\varepsilon_2$,
	so that its gradient is given by $\nabla_{\bm\varepsilon}\big(h_{\v_1^{\f}}\circ G_{\f}\big)=(-1;-1)$.
	Note that the chain rule guarantees that
	\begin{equation}
		\label{gradeq}
		\nabla_{\bm\varepsilon}\big(h_{\v_1^{\f}}\circ G_{\f}\big) = J_{\f}^{\top}\nabla h_{\v_1^{\f}},
	\end{equation}
	since $J_{\f} = (\v^{\f}_2-\v^{\f}_1,\v^{\f}_3-\v^{\f}_1)$ is the Jacobian matrix of $G_{\f}$.
	Since the basis function $h_{\v_1^{\f}}$ is linear on $\f$, its gradient $\nabla h_{\v_1^{\f}}$ lies in the plane spanned by $\v^{\f}_2-\v^{\f}_1 $ and $\v^{\f}_3-\v^{\f}_1$, namely in $\rge(J_{\f})$ (see \cite[Section 3.1]{yjjmanifold} for details).
	Therefore, we know from \eqref{gradeq} that $\nabla h_{\v_1^{\f}} = (J_{\f}^{\top})^\dag(-1;-1)$. Similar discussions on $h_{\v_2^{\f}}$ and $h_{\v_3^{\f}}$ ensure \eqref{eq:nabla-h}. 
	
	\small
	\bibliographystyle{plain}
	\bibliography{ref.bib}

\end{document}